\documentclass[10pt]{article}

\input{ptrace.sty}

\title{Partially traced categories\footnote{This research was
    supported by NSERC.}}
\author{Octavio Malherbe, Philip J. Scott and Peter Selinger}
\date{}

\begin{document}
\maketitle

\begin{abstract}
This paper deals with
questions relating to Haghverdi and Scott's notion of partially traced
categories. The main result is a representation theorem for such
categories: we prove that every partially traced category can be
faithfully embedded in a totally traced category. Also conversely,
every symmetric monoidal subcategory of a totally traced category is partially
traced, so this characterizes the partially traced categories
completely. The main technique we use is based on Freyd's
paracategories, along with a partial version of Joyal, Street, and
Verity's $\Int$-construction.

\end{abstract}

\section{Introduction}
Partially traced monoidal categories were introduced by Haghverdi and
Scott {\cite{HS05a,HS09}} as a general framework for modelling a typed
categorical version of Girard's Geometry of Interaction.  The Geometry
of Interaction (GoI) was developed by Girard in a series of
influential works that examine dynamical models of proofs in linear
logic and their evaluation under normalization, using operator
algebras and functional analysis
{\cite{GIRARD88,GIRARD95,GIRARD03,GIRARD011}}.  This program has
recently received increased attention as also having connections with
quantum computation and quantum protocols
{\cite{AbcoeckeA,Sel2004,Sel04b}}.

One of the objectives of this article is to systematically explore the Haghverdi-Scott
notion of partially traced categories by providing a representation
theorem which establishes a precise correspondence between partially
traced and totally traced categories. 

\section{Background}

To fix the notation for this paper, we briefly recall some basic
notions from monoidal category theory. For more details, see
e.g.~\cite{Borceux94, LambekScott86, MacLane91}.

\subsection{Monoidal categories}
\begin{definition}
  \rm A \textit{monoidal} category, also sometimes called
  \textit{tensor} category, is a category $\cC$ with a bifunctor
  $\otimes:\cC\times\cC\rightarrow \cC$ together with a unit object
  $I\in\cC$ and natural isomorphisms $\rho_A:A\otimes
  I\stackrel{\cong}\rightarrow A$, $\lambda_A:I\otimes
  A\stackrel{\cong}\rightarrow A$, and $\alpha_{A,B,C}:A\otimes
  (B\otimes C)\stackrel{\cong}\rightarrow (A\otimes B)\otimes C$,
  satisfying some coherence axioms~\cite{MacLane91}. The monoidal
  category is {\em strict} if $\rho$, $\lambda$, and $\alpha$ are
  identities. It is well-known that every monoidal category is
  equivalent to a strict one {\cite{MacLane91}}. Here, and throughout
  the paper, we often omit the subscripts from notations such as
  $\alpha_{A,B,C}$ when they are clear from the context.
\end{definition}

\begin{definition}
\rm
A \textit{symmetric} monoidal category consists of a monoidal category $\cC$ with a chosen natural isomorphism $\sigma_{A,B}:A\otimes B\stackrel{\cong}\rightarrow B\otimes A$, called {\em symmetry}, which satisfies $\sigma_{B,A}\circ\sigma_{A,B}=\id_{A\x B}$, $\lambda_A\circ\sigma_{A,I}=\rho_A$, and $\alpha_{C,A,B}\circ \sigma_{A\x B,C}\circ\alpha_{A,B,C} = (\sigma_{A,C}\x\id_B)\circ\alpha_{A,C,B}\circ(\id_A\x\sigma_{B,C})$.
\end{definition}

\begin{definition}\label{def-monoidal-functor}
\rm
A \textit{monoidal functor} $(F,m_{A,B},m_I)$ between monoidal categories $(\cC,\otimes,I,\alpha,\rho,\lambda)$ and $(\cD,\otimes',I',\alpha',\rho',\lambda')$ is a functor $F:\cC\rightarrow\cD$ equipped with:
\begin{itemize}
\item[-]
morphisms $m_{A,B}:F(A)\otimes'F(B)\rightarrow F(A\otimes B)$ natural in $A$ and $B$,
\item[-] a morphism $m_I:I'\rightarrow F(I)$,
\end{itemize}
which satisfy some coherence axioms preserving the symmetric monoidal
structure~\cite{MacLane91}.  A monoidal functor is \textit{strong}
when $m_I$ and all the $m_{A,B}$ are isomorphisms. It is
\textit{strict} when $m_I$ and all the $m_{A,B}$ are identities.

If in addition, $\cC$ and $\cD$ are symmetric monoidal with respective
symmetries $\sigma$ and $\sigma'$, then $F$ is a {\em symmetric}
monoidal functor if for all $A,B$,
\[\xymatrix@=20pt{
FA\otimes' FB \ar[r]^{\sigma'}  \ar[d]_{m}& FB\otimes'FA \ar[d]^{m}  \\
     F(A\otimes B)  \ar[r]_{F(\sigma)}                       &  F(B\otimes A).
}\]
\end{definition}

\subsection{Traced monoidal categories}\label{ssec-traced-monoidal-categories}
Traced monoidal categories were introduced by Joyal, Street and Verity
\cite{JSV96} as an attempt to model an abstract notion of trace
arising in different fields of mathematics, such as algebraic
topology, knot theory, and theoretical computer science. In computer
science, this abstraction has been particularly useful in the
description of feedback, fixed-point operators, the execution formula
in Girard's Geometry of Interaction \cite{GIRARD88}, etc.

\begin{definition}\label{def-traced-monoidal-cat}
\rm A \textit{trace} for a symmetric monoidal category $(\cC,\otimes,I,\rho,\lambda,\sigma)$ consists of a family of functions
\[\Tr^U_{A,B}:\cC(A\otimes U,B\otimes U)\rightarrow \cC(A,B),\]
natural in $A$, $B$, and dinatural in $U$, satisfying the following axioms. Here we write without loss of generality as if $\cC$ were strict.

\begin{itemize}
\item\textbf{Strength:}
For all $f:A\otimes U\rightarrow B\otimes U$ and $g:C\rightarrow D$,
\[ g\otimes \Tr^{U}_{A,B}(f)= \Tr^{U}_{C\otimes A,D\otimes B}(g\otimes f).
\]
\item\textbf{Vanishing I:}
For all $f:A\x I\to B\x I$,
\[ f = \Tr^{I}_{A,B}(f).
\]
\item\textbf{Vanishing II:}
For all $f:A\x U\x V\to B\x U\x V$,
\[ \Tr^{U}_{A,B}(\Tr^{V}_{A\otimes U,B\otimes U}(f))=
\Tr^{U\otimes V}_{A,B}(f).
\]
\item\textbf{Yanking:}
For all $A$,
\[ \Tr^{A}_{A,A}(\sigma_{A,A})=1_{A}.
\]
\end{itemize}
Because we need them later, we explicitly spell out the conditions of
naturality and dinaturality:
\begin{itemize}
\item\textbf{Naturality:}
For all $f:A\x U\to B\x U$, $g:A'\rightarrow A$, and $h:B\rightarrow B'$,
\[ h \circ \Tr^{U}_{A,B}(f)\circ g=\Tr^{U}_{A',B'}((h\otimes 1_{U})\circ f\circ(g\otimes 1_{U})).
\]
\item\textbf{Dinaturality:}
For all $f:A\otimes U\rightarrow B\otimes U'$ and $g:U'\rightarrow U$,
\[ \Tr^{U}_{A,B}((1_{B}\otimes g)\circ f)=\Tr^{U'}_{A,B}(f\circ (1_{A}\otimes g)).
\]
\end{itemize}
\end{definition}

\begin{definition}
\rm
Let $(\cC,\Tr)$ and $(\cD,\Tr')$ be traced monoidal categories.
We say that a strong symmetric monoidal functor $(F,m):\cC\rightarrow\cD$ is \textit{traced monoidal} when it preserves the trace operator in the following sense:
for all $f:A\otimes U\rightarrow B\otimes U$,
\[ {\Tr'}^{FU}_{FA,FB}(m^{-1}_{A,U}\circ F(f)\circ m_{A,U})=F(\Tr^U_{A,B}(f))
~:~ FA\rightarrow FB.
\]
\end{definition}

\subsection{Graphical language}\label{ssec-graphical-total}

\begin{table}
  \begin{center}
    \begin{tabular}{@{}l@{}}
      Object $A$:\\
      \tmath{\wirechart{@C=.7cm}{*{}\wireright{rr}{A}&&}} \\\\
      Morphism $f:$ {\small $A_1\,{\x}\,\ldots\,{\x}\, A_n\,{\to}\, B_1\,{\x}\,\ldots\,{\x}\, B_m$}: \\\\
      \tmath{\wirechart{@C=1.5cm@R=0.4cm}{
        *{}\wireright{r}{A_n}&
        \blank\wireright{r}{B_m}&
        \\
        *{}\wireright{r}{A_1}^<>(.8){\vdots}&
        \blank\ulbox{[].[u]}{f}\wireright{r}{B_1}^<>(.2){\vdots}&
        \\
      }}\\\\
      Composition  $g\circ f:A\to C$: \\\\
      \tmath{\wirechart{}{*{}\wireright{r}{A}&
        \blank\ulbox{[]}{f}\wireright{r}{B}&
        \blank\ulbox{[]}{g}\wireright{r}{C}&
      }}
    \end{tabular}
    \begin{tabular}{@{}llc@{}}
      Tensor  $f\x g:A\x C\to B\x D$: \\\\
      \tmath{\wirechart{@C=1.5cm@R=0.4cm}{
          \vsblank\wireright{r}{C}&\blank\ulbox{[]}{g}\wireright{r}{D}&\\
          \vsblank\wireright{r}{A}&\blank\ulbox{[]}{f}\wireright{r}{B}&\\
        }}\\\\
      Symmetry  $\sym_{A,B}:A\x B\to B\x A$: \\\\
      \tmath{\wirechart{@C=1.5cm@R=0.4cm}{
          *{}\wireright{r}{B}&\blank\wirecross{d}\wireright{r}{A}&\\
          *{}\wireright{r}{A}&\blank\wirecross{u}\wireright{r}{B}& }}\\\\
      Trace  $\Tr^U_{A,B} f:A\to B$: \\\\
      \tmath{\wirechart{@C=1cm@R=.5cm}{
          \blank\wireopen{d}\wire{rr}{}&&
          \blank\wireclose{d}
          \\
          \blank\wireright{r}{U}&
          \blank\wireright{r}{U}&
          \blank
          \\
          \wireid\wireright{r}{A}&
          \blank\ulbox{[].[u]}{f}\wireright{r}{B}&
          \wireid
          \\}}
    \end{tabular}
  \end{center}
  \caption{The graphical language of traced monoidal categories}
  \label{tab-graphical-traced}
\end{table}

Graphical calculi are a useful tool for reasoning about monoidal
categories, dating back at least to the work of Penrose~\cite{Pen}.
There are various graphical languages that are provably sound and
complete for equational reasoning in different kinds of monoidal
categories. They allow efficient geometrical and topological insights
to be used in a kind of calculus of ``wirings'', which simplifies
diagrammatic reasoning. See~\cite{Sel2009} for a detailed survey of
such graphical languages.

In particular, there is a graphical language for traced monoidal
categories, which was already used in Joyal, Street, and Verity's
original paper~\cite{JSV96}. As shown in
Table~\ref{tab-graphical-traced}, wires represent objects, boxes
represent morphisms, composition is represented by connecting the
outgoing wires of one diagram to the incoming wires of another, and
tensor product is represented by stacking wires and boxes
vertically. Finally, trace is represented by a loop.  The axioms of
traced (symmetric) monoidal categories are illustrated in
Table~\ref{tab-graphical-axioms}.

\begin{table}

\begin{center}
\begin{tabular}{l}
  Naturality:\\\\
  \tmath{
    \raisebox{-2mm}{\includegraphics[scale=0.65]{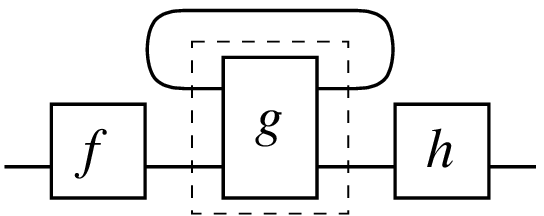}}=
    \raisebox{-2mm}{\includegraphics[scale=0.65]{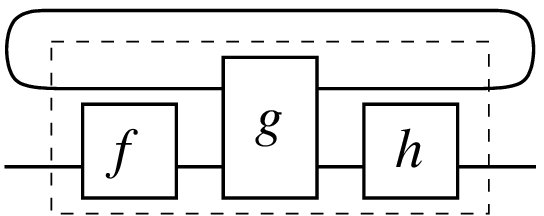}}}\\\\

  Dinaturality:\\\\
  \tmath{
    \raisebox{-2mm}{\includegraphics[scale=0.65]{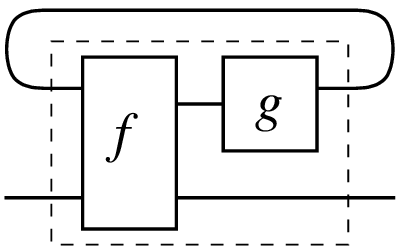}}=
    \raisebox{-2mm}{\includegraphics[scale=0.65]{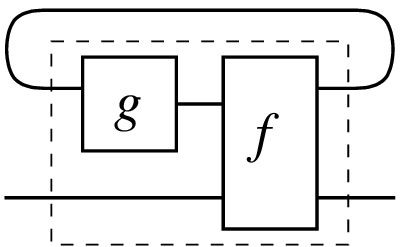}}}\\\\

  Strength:\\\\
  \tmath{
    \raisebox{-6.5mm}{\includegraphics[scale=0.65]{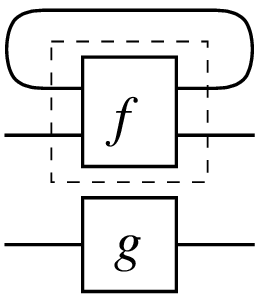}}=
    \raisebox{-7.4mm}{\includegraphics[scale=0.65]{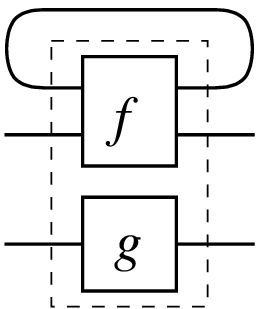}}}
\end{tabular}\begin{tabular}{l}
  Vanishing I:\\\\
  \tmath{
    \raisebox{-2mm}{\includegraphics[scale=0.65]{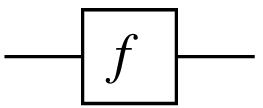}}=
    \raisebox{-2mm}{\includegraphics[scale=0.65]{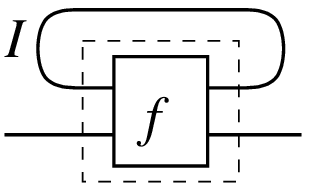}}}\\\\

  Vanishing II:\\\\
  \tmath{
    \raisebox{-3mm}{\includegraphics[scale=0.65]{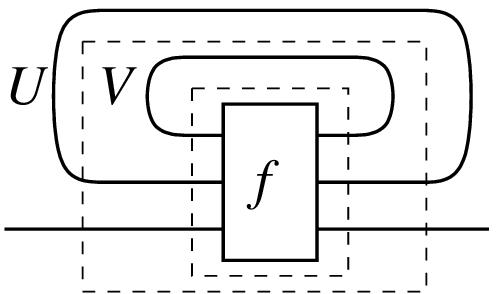}}=
    \raisebox{-2mm}{\includegraphics[scale=0.65]{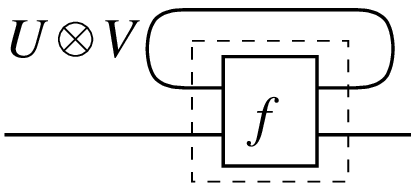}}}\\\\

  Yanking:\\\\
  \tmath{
    \raisebox{-1mm}{\includegraphics[scale=0.65]{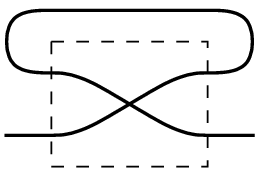}}=
    ~\raisebox{1mm}{\includegraphics[scale=0.65]{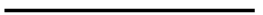}}}
\end{tabular}
\end{center}

\caption{The axioms of traced monoidal categories}
\label{tab-graphical-axioms}
\end{table}

We say that a diagram is {\em expanded} if all its wires are labelled
by object variables and all its boxes are labelled by morphism
variables (as opposed to composite object and morphism terms). Thus,
for example, a wire labelled $A\x B$ is not expanded, but a pair of
wires labelled $A$ and $B$ is expanded. Each non-expanded diagram can
be converted to an equivalent expanded diagram.  The following theorem
shows the validity of diagrammatic reasoning in traced monoidal
categories.

\begin{theorem}[Coherence, see~\cite{Sel2009}]
  A well-formed equation between morphisms in the language of
  symmetric traced monoidal categories follows from the axioms of
  symmetric traced monoidal categories if and only if it holds, up to
  isomorphism of expanded diagrams, in the graphical language.
\end{theorem}

\subsection{Compact closed categories}

\begin{definition}
  \rm A \textit{compact closed} category is a symmetric monoidal
  category $\cC$ in which for every object $A$, there is a given
  object $A^{*}$, called the {\em dual} of $A$, and a given pair of
  arrows $\eta:I\rightarrow A^{*}\otimes A$ (called the {\em unit}),
  $\eps:A\otimes A^{*}\rightarrow I$ (called the {\em counit})
  such that the following are identities:
  \[  A \catarrow{\rho^{-1}} A\x I \catarrow{\id\x\eta} A\x(A^*\x A)
  \catarrow{\alpha} (A\x A^*)\x A\catarrow{\eps\x\id} I\x A
  \catarrow{\lambda} A ~=~ \id_A,
  \]
  \[  A^* \catarrow{\lambda^{-1}} I\x A^* \catarrow{\eta\x\id} (A^*\x A)\x A^*
  \catarrow{\alpha^{-1}} A^*\x(A\x A^*) \catarrow{\id\x\eps}
  A^*\x I\catarrow{\rho} A^* ~=~ \id_{A^*}.
  \]
\end{definition}

\begin{proposition}\label{prop-canonical-trace}
  Let $\cC$ be a compact closed category. Then $\cC$ has a unique
  trace, which we call the {\em canonical trace}. It is defined as
  follows (here we write without loss of generality as if the category
  were strict):
  \[ \Tr^U_{A,B}(f) ~=~
  A \catarrow{\id\x\eta} A\x U^*\x U
  \catarrow{\id\x\sigma} A\x U\x U^*
  \catarrow{f\x\id} B\x U\x U^*
  \catarrow{\id\x\eps} B.
  \]
  Moreover, every strong symmetric monoidal functor between compact
  closed categories preserves the compact closed structure, and
  therefore the canonical trace.
\end{proposition}

\begin{proof}
See~\cite{JSV96}. For uniqueness of the trace, see
\cite[Appendix B]{Has08}.
\end{proof}

\section{Partially traced categories}

Partially traced symmetric monoidal categories were introduced by
Haghverdi and Scott \cite{HS05a} as part of a categorical framework
for a typed version of the Geometry of Interaction (GoI).  

An important aspect of modelling the dynamics of proofs in Girard's
concrete models of GoI is that proofs are interpreted as operators,
and cut-elimination (normalization) is interpreted in terms of
feedback (the ``execution formula'').  Haghverdi and Scott
\cite{HS05a} used a partial trace to define a categorical version of
the execution formula.  This execution formula is (for large classes
of sequents) an invariant of the cut-elimination process.  Types are
given by an abstract orthogonality relation in the sense of Hyland and
Schalk \cite{HylSch03}. Such an orthogonality relation arises
naturally in the partially traced setting, and yields the required
convergence properties of Girard's execution formula.  Thus, partially
traced categories (with additional structure) provide the necessary
ingredients for running Girard's GoI machinery.

We note that, while totally traced categories are a special case of
partially traced categories, partiality was important in constructing
nontrivial types in the typed version of GoI in {\cite{HS05a}}.
Indeed, if one assumes a total trace in this setting, the type
structure collapses. By contrast, the earlier analysis of GoI in
{\cite{HS06a}} used a {\em total} categorical trace, but required the
category to be equipped with a reflexive object satisfying appropriate
domain equations, which is a very strong assumption.

In this section, we recall the definition of a partially traced
category, and give some examples. We also show that each symmetric
monoidal subcategory of a (totally or partially) traced category is
partially traced, which gives rise to many more examples.

\subsection{Definition of partially traced categories}

We recall the definition of a partially traced (symmetric monoidal) category from \cite{HS05a}. We begin with some notation for partial functions.

\begin{definition}\label{def-kleene-equality}
 Let $f$ and $g$ be partially defined operations. We write
 $f(x)\defined$ if $f(x)$ is defined, and $f(x)\undefined$ if it is
 undefined. Following Freyd and Scedrov~\cite{FREYD-SCEDROV-Categories-Allegories},
 we also write $f(x)\kleeneeq g(x)$ if $f(x)$ and $g(x)$ are either
 both undefined, or else they are both defined and equal. The relation
 ``$\kleeneeq$'' is known as {\em Kleene equality}. We also write
 $f(x)\kleeneleq g(x)$ if either $f(x)$ is undefined, or else $f(x)$ and
 $g(x)$ are both defined and equal. The relation ``$\kleeneleq$'' is
 known as {\em directed Kleene equality}.
\end{definition}

\begin{definition}[\cite{HS05a,HS09}]\label{def-partial-trace}
  Suppose $(\cC,\otimes,I,\rho,\lambda,\sigma)$  is a symmetric monoidal
  category. A {\em partial trace} is given by a family of partial
  functions $\Tr^U_{A,B} : \cC(A\otimes U, B\otimes U) \rightharpoonup
  \cC(A,B)$, satisfying the following axioms. Once again, we write the
  axioms as if $\cC$ were strict.
\begin{itemize}
\item\textbf{Naturality:}
For all $f:A\x U\to B\x U$, $g:A'\rightarrow A$, and $h:B\rightarrow B'$,
\[ h \circ \Tr^{U}_{A,B}(f)\circ g\kleeneleq\Tr^{U}_{A',B'}((h\otimes 1_{U})\circ f\circ(g\otimes 1_{U})).
\]
\item\textbf{Dinaturality:}
For all $f:A\otimes U\rightarrow B\otimes U'$ and $g:U'\rightarrow U$,
\[ \Tr^{U}_{A,B}((1_{B}\otimes g)\circ f)\kleeneeq\Tr^{U'}_{A,B}(f\circ (1_{A}\otimes g)).
\]
\item\textbf{Strength:}
For all $f:A\otimes U\rightarrow B\otimes U$ and $g:C\rightarrow D$,
\[ g\otimes \Tr^{U}_{A,B}(f)\kleeneleq \Tr^{U}_{C\otimes A,D\otimes B}(g\otimes f).
\]
\item\textbf{Vanishing I:}
For all $f:A\x I\to B\x I$,
\[ f\kleeneeq \Tr^{I}_{A,B}(f).
\]
\item\textbf{Vanishing II:}
For all $f:A\x U\x V\to B\x U\x V$,
\[ \Tr^V_{A\otimes U,B\otimes U}(f)\,\defined \quad\mbox{implies}\quad
\Tr^{U}_{A,B}(\Tr^{V}_{A\otimes U,B\otimes U}(f))\kleeneeq
\Tr^{U\otimes V}_{A,B}(f).
\]
\item\textbf{Yanking:}
For all $A$,
\[ \Tr^{A}_{A,A}(\sigma_{A,A})\kleeneeq 1_{A}.
\]
\end{itemize}
A {\em partially traced} category is a symmetric monoidal
 category with a partial trace.
\end{definition}

Note that in the vanishing I axiom, the left-hand side is always
defined, so Kleene equality in this case just means that
$\Tr^{I}_{A,B}(f)$ is always defined an equals $f$. A similar remark
applies to the yanking axiom. 

\begin{remark}
  Comparing this to Definition~\ref{def-traced-monoidal-cat}, we see that a
  traced monoidal category is just a partially traced category where
  the trace operation happens to be total. We sometimes refer to
  traced monoidal categories as {\em totally} traced categories, when
  we want to emphasize that they are not partial.
\end{remark}

 \begin{definition}
 The subset of $\cC(A\otimes U,B\otimes U)$ where
 $\Tr^U_{A,B}$ is defined is sometimes called the {\em trace class}, and
 is written
 \[
 \Trc^U_{A,B} = \s{ f:A\otimes U\rightarrow B\otimes U \mid
   \Tr^U_{A,B}(f)\,\defined }.
 \]
\end{definition}

\begin{remark}\label{rem-naturality}
  In case $g$ and $h$ are isomorphisms, by naturality we have
  \[
  h^{-1} \circ \Tr^{U}_{A,B}(f')\circ g^{-1}\kleeneleq\Tr^{U}_{A',B'}((h^{-1}\otimes
  1_{U})\circ f'\circ(g^{-1}\otimes 1_{U})),
  \]
  and therefore
  \[\Tr^{U}_{A',B'}((h\otimes 1_{U})\circ f\circ(g\otimes 1_{U}))
  \kleeneleq
  h \circ \Tr^{U}_{A,B}(f)\circ g,
  \]
  where $f'=(h\otimes 1_{U})\circ f\circ(g\otimes 1_{U})$. Therefore,
  the naturality axiom holds with Kleene equality when $g$ and $h$ are
  isomorphisms.
\end{remark}

\begin{remark}
  The precondition to the vanishing II axiom is redundant for the
  left-to-right direction. In other words, we have the directed Kleene
  equality
  \[ \Tr^{U}_{A,B}(\Tr^{V}_{A\otimes U,B\otimes U}(f))\kleeneleq
  \Tr^{U\otimes V}_{A,B}(f)
  \]
  regardless of whether $\Tr^{V}_{A\otimes U,B\otimes U}(f)$ is
  defined or not. However, the assumption $\Tr^{V}_{A\otimes
    U,B\otimes U}(f)\defined$ is of course critical for the
  right-to-left direction.
\end{remark}

\begin{lemma}
  The strength axiom in the context of
  Definition~\ref{def-partial-trace} is equivalent to the following
  axiom (see also~\cite{JSV96}):\rm
  \begin{itemize}
  \item\textbf{Superposing:} For all $f:A\otimes U\rightarrow B\otimes U$
    and $g:C\rightarrow D$,
    \[ \Tr^{U}_{A,B}(f)\otimes g\kleeneleq  \Tr^{U}_{A\otimes
      C,B\otimes D}((1_B\otimes \sigma_{U,D})\circ (f\otimes g)\circ
    (1_A\otimes \sigma_{C,U})).
    \]
  \end{itemize}
\end{lemma}

\begin{proof}
  By the axioms of symmetric monoidal categories, we have $(1_B\otimes
  \sigma_{U,D})\circ (f\otimes g)\circ (1_A\otimes \sigma_{C,U}) =
  (\sigma\x\id_U)\circ(g\x f)\circ(\sigma\x\id_U)$. From this and
  Remark~\ref{rem-naturality}, it follows that the right-hand sides of
  the superposing and strength axioms are related by conjugation with
  $\sigma$:
  \begin{eqnarray*}
    \Tr^{U}_{A\otimes C,B\otimes D}((1_B\otimes \sigma_{U,D})\circ
    (f\otimes g)\circ (1_A\otimes \sigma_{C,U}))
    &\kleeneeq& \Tr^{U}_{A\otimes
      C,B\otimes D}((\sigma\x\id_U)\circ(g\x f)\circ(\sigma\x\id_U))
    \\
    &\kleeneeq&
    \sigma\circ \Tr^{U}_{C\otimes A,D\otimes B}(g\x f)\circ \sigma.
  \end{eqnarray*}
  On the other hand, the left-hand sides of the superposing and
  strength axioms are also related by conjugation with $\sigma$:
  \begin{eqnarray*}
    \Tr^{U}_{A,B}(f)\otimes g
    &\kleeneeq&
    \sigma\circ(g\x \Tr^{U}_{A,B}(f))\circ\sigma.
  \end{eqnarray*}
  The claim then follows.
\end{proof}

\subsection{Graphical language}\label{ssec-graphical-partial}

Because a morphism such as
\[ \vcenter{\wirechart{@C=.5cm@R=.5cm}{
          \blank\wireopen{d}\wire{rr}{}&&
          \blank\wireclose{d}
          \\
          \blank\wireright{r}{}&
          \blank\wireright{r}{}&
          \blank
          \\
          \wireid\wireright{r}{}&
          \blank\ulbox{[].[u]}{f}\wireright{r}{}&
          \wireid}}
\]
may be undefined in a partially traced category, we may not a priori
assume that the graphical language of
Section~\ref{ssec-graphical-total} is sound for partially traced
categories, or even that every diagram describes a unique morphism.
For example, both sides of the naturality axiom correspond, up to
isomorphism of diagrams, to the same diagram
\[ \vcenter{\wirechart{@C=.5cm@R=.5cm}{
          &\blank\wireopen{d}\wire{rr}{}&&
          \blank\wireclose{d}
          \\
          &\blank\wireright{r}{}&
          \blank\wireright{r}{}&
          \blank
          \\
          *{}\wireright{r}{}&\blank\ulbox{[]}{g}\wireright{r}{}&
          \blank\ulbox{[].[u]}{f}\wireright{r}{}&
          \blank\ulbox{[]}{h}\wireright{r}{}&
          *{~.}
          }}
\]
However, one side of the axiom may be undefined and the other defined,
so the diagram does not have a unique meaning.

Nevertheless, we wish to use graphical reasoning to simplify our
exposition, particularly in Section~\ref{sec-pint}. The following
standard trick will allow us to do this. Whenever we take the trace of
a composite diagram, we will draw a special box around the portion of
the diagram that is being traced, like this:
\[ \vcenter{\wirechart{@C=.5cm@R=.5cm}{
          \blank
          \wireopen{d}\wire{rrr}{}&&&
          \blank\wireclose{d}
          \\
          \blank\wireright{rr}{}&&
          \blank\wireright{r}{}&
          \blank
          \\
          *{}\wireright{rr}{}&&
          \blank\ulbox{[].[u]}{f}\wireright{r}{}&
          \blank\ulbox{[]}{h}\wireright{r}{}&
          *{}\\
          *{}\wireright{r}{}&\blank
          \dashbox{[].[uu].[r]}{!U+<3mm,\addheighthalf>!}
          \ulbox{[]}{g}\wireright{r}{}&
          \blank\ulbox{[]}{k}\wireright{rr}{}&&
          *{~.}\\
          }}
\]
Note that, since every partially traced category is a
symmetric monoidal (total) category, the graphical language of symmetric monoidal
categories is still {\em sound} for partially traced categories, and
therefore any symmetric monoidal portion of a graphical diagram can be
soundly manipulated up to graph isomorphism. This means that one can
soundly manipulate the inside of a box, as well as the box as a whole,
up to graph isomorphism. With this convention, any diagram (up to
box-respecting graph isomorphism) has a unique meaning (up to Kleene
equality) in a partially traced category.

\subsection{Examples: partial traces on $(\Vect,\oplus)$}

It is well-known that the category $\Vectfin$ of finite dimensional
vector spaces (over any field $k$), with the symmetric monoidal
structure given by the tensor product $\otimes$, is totally traced. In
fact, this category is compact closed.

On the other hand, with respect to the monoidal structure given by
the biproduct $\oplus$, neither $\Vect$ nor $\Vectfin$ is totally traced.
However, it is possible to define a partial trace on these
categories. In fact, this can be done in more than one way, as we will
now discuss.

Recall that in a category with biproducts, a morphism $f: A\oplus
U\rightarrow B\oplus U$ is characterized by the matrix $
\xmatrix{.5}{.5}{ll}{
  f_{11} & f_{12}\\
  f_{21} & f_{22}}, $ where $f_{ij}=\pi_{i} \comp f\comp
\inj_{j}$. Composition corresponds to multiplication of matrices. Also
recall that an {\em additive category} is a category with finite
biproducts and where each morphism $f:A\to B$ has an additive inverse
$g:A\to B$ such that $f+g=0$.

\subsubsection{Non-examples: Kleene trace and sum trace}

A first attempt to define a partial trace with respect to biproducts
on vector spaces is by summing over all paths in the graph
\[ \includegraphics[width=1.8in]{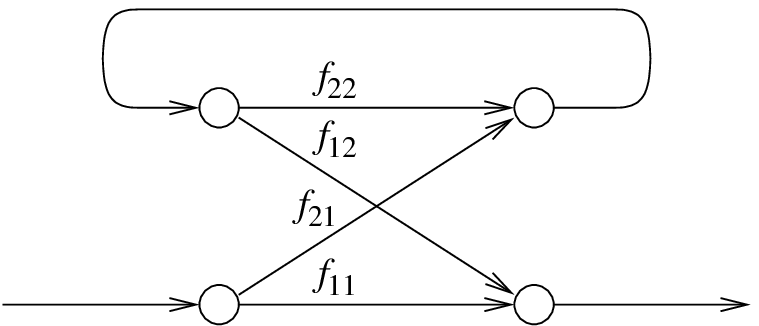}.
\]
We consider two variants:

\begin{definition}[Kleene trace {\cite{Plotkin03}}]\label{def-example-kleene}
  The {\em Kleene trace} is the following partial operation on
  $(\Vect,\oplus)$. For $f:A\oplus U\rightarrow B\oplus U$, define
  \begin{equation}\label{eqn-example-kleene}
    \TrK^U(f)=f_{11}+f_{12}(\sum_{n\geq 0}f_{22}^n)f_{21},\\
  \end{equation}
  if this sum exists, and $\TrK^U(f)$ is undefined otherwise.
\end{definition}

To give the summation an unambiguous meaning, let us assume here that
the vector spaces are over the real or complex numbers, and that
convergence is taken with respect to some convenient topology, such as
the weak operator topology, where $X_n\rightarrow X$ iff for all $v\in
A$ and $w\in B^*$, $wX_nv\rightarrow wXv$. We also consider:

\begin{definition}[Sum trace]\label{def-example-sum}
  On $(\Vect,\oplus)$, for $f:A\oplus U\rightarrow B\oplus U$, define
  the {\em sum trace}
  \begin{equation}\label{eqn-example-sum}
    \TrS^U(f)=f_{11}+\sum_{n\geq 0}(f_{12}\,f_{22}^n\,f_{21}),\\
  \end{equation}
  if this sum exists, and $\TrS^U(f)$ is undefined otherwise.
\end{definition}

\begin{proposition}
  Neither (\ref{eqn-example-kleene}) nor (\ref{eqn-example-sum}) is a
  partial trace in the sense of Definition~\ref{def-partial-trace}.
  Both operations satisfy naturality, dinaturality, strength,
  vanishing I, and yanking. However, both fail to satisfy vanishing
  II.
\end{proposition}

\begin{proof}
  Naturality, dinaturality, strength, vanishing I, and yanking are
  straightforward to check. To show that the sum trace does not
  satisfy vanishing II, consider $A=B=U=V=k$ and $f:A\oplus U\oplus
  V\to B\oplus U\oplus V$ given by
  \[ \zmatrix{ccc}{
      1 & 1 & 0 \\
      1 & -2 & 1 \\
      0 & 1 & 1/2}.
  \]
  Then
  \[ \TrS^V f = \zmatrix{cc}{1 & 1 \\ 1 & -2}
  + \sum_{n\geq 0} \zmatrix{c}{0 \\ 1} \left(\frac{1}{2}\right)^n
  \zmatrix{cc}{0 & 1}
  = \zmatrix{cc}{1 & 1\\ 1 & 0}.
  \]
  In particular, this sum exists, so the hypothesis of vanishing II is
  satisfied. Now $\TrS^U \TrS^V f$ exists and is equal to
  \[ \TrS^U \TrS^V f = 1 + \sum_{m\geq 0} 1\cdot 0^m\cdot 1 = 2.
  \]
  On the other hand,
  \[ \TrS^{U\oplus V} f = 1 + \sum_{n\geq 0} \zmatrix{cc}{1 & 0}
  \zmatrix{cc}{-2 & 1 \\ 1 & 1/2}^n \zmatrix{c}{1 \\ 0},
  \]
  which does not converge, contradicting the vanishing II axiom. The
  same counterexample also applies to the Kleene trace.
\end{proof}

\subsubsection{Haghverdi and Scott's partial trace on $(\Vect,\oplus)$}

One of the motivating examples of a partially traced category in
{\cite{HS05a,HS06,HS09}} is the following partial trace on
$(\Vect,\oplus)$. It can be seen as an effort to make the Kleene trace
(\ref{eqn-example-kleene}) more often defined by replacing the sum
$\sum_{n\geq 0}f_{22}^n$ by its limit $(I-f_{22})^{-1}$. The following
definition makes sense in finite or infinite dimensions and over any
base field, or indeed in any additive category.

\begin{definition}[Haghverdi-Scott trace {\cite{HS05a}}]\label{def-hs-trace}
  \rm On $(\Vect,\oplus,\textbf{0})$, or on any additive category, we
  define a partial trace as follows: for $f:A\oplus U\rightarrow
  B\oplus U$, let
  \begin{equation}\label{eqn-hs-trace}
    \TrHS^U(f)=f_{11}+f_{12}(I-f_{22})^{-1}f_{21},
  \end{equation}
  if $I-f_{22}$ is invertible, and $\TrHS^U(f)$ is undefined
  otherwise. Here, $I=\id:U\rightarrow U$ is the identity map on $U$.
\end{definition}

\begin{proposition}
  The Haghverdi-Scott trace is a partial trace.
\end{proposition}

\begin{proof}~\cite{HS05a,HS06,HS09}.
\end{proof}

\begin{remark}
  Both the sum trace and the Haghverdi-Scott trace are strictly more
  defined than the Kleene trace, in the sense that for all $f$,
  $\TrK^U(f)\kleeneleq \TrS^U(f)$ and $\TrK^U(f)\kleeneleq
  \TrHS^U(f)$. Moreover, it appears that when the sum trace and the
  Haghverdi-Scott trace are both defined, then they
  coincide\footnote{We know this for certain only in the finite
    dimensional case.}. However, the sum trace and the Haghverdi-Scott
  trace can each be defined without the other being defined. For
  example, for $f=\xmatrix{.5}{.5}{cc}{1&0\\0&1}$, the sum trace is
  defined but the Haghverdi-Scott trace is not, whereas for
  $f=\xmatrix{.5}{.5}{cc}{1&1\\1&2}$, the Haghverdi-Scott trace is
  defined and the sum trace is not. In fact, as the following
  proposition shows, there is no partial trace (and hence definitely
  no total trace) on $(\Vect,\oplus,\textbf{0})$ that simultaneously
  generalizes the sum trace and the Haghverdi-Scott trace.
\end{remark}

\begin{proposition}\label{prop-trace-paradox}
  There exists no partial trace $\Tr$ on $(\Vect,\oplus)$, such
  that for all $f:A\oplus U\rightarrow B\oplus U$,
  \[ \TrS^U(f) \kleeneleq \Tr^U(f)
  \quad\mbox{and}\quad
  \TrHS^U(f) \kleeneleq \Tr^U(f).
  \]
\end{proposition}

\begin{proof}
  Suppose there is such a partial trace $\Tr$. Let $A=B=U=k$,
  $X=\xmatrix{.5}{.5}{cc}{0&1\\1&0}$ and consider $f,g:A\oplus U\oplus
  U\to B\oplus U\oplus U$ given by
  \[ f = \zmatrix{ccc}{0& 1& 1 \\ 0 & 2 & 1 \\ 1 & -1 & 0}
  \quad\mbox{and}\quad
  g = (\id\oplus X)f(\id\oplus X^{-1})
  = \zmatrix{ccc}{0& 1& 1 \\  1 & 0& -1 \\0 & 1& 2}.
  \]
  By direct calculation, one can verify that both $\TrS^U\TrHS^U(f)$
  and $\TrS^U\TrHS^U(g)$ are defined and
  \[ \TrS^U\TrHS^U(f) = 1 \quad\mbox{and}\quad \TrS^U\TrHS^U(g) = 0.
  \]
  By assumption,
  \[ \Tr^U\Tr^U(f) = 1 \quad\mbox{and}\quad \Tr^U\Tr^U(g) = 0,
  \]
  hence by vanishing II,
  \[ \Tr^{U\oplus U}(f) = 1 \quad\mbox{and}\quad \Tr^{U\oplus U}(g) = 0.
  \]
  On the other hand, dinaturality requires $\Tr^{U\oplus
    U}(f)\kleeneeq \Tr^{U\oplus U}(g)$, a contradiction.

\end{proof}

\subsubsection{The kernel-image partial trace on $(\Vect,\oplus)$}
\label{sssec-ki-trace}

The following definition generalizes the Haghverdi-Scott partial
trace, and is defined on slightly more morphisms.

\begin{definition}[Kernel-image trace]\label{def-kernel-image}
  We define another partial trace on $(\Vect,\oplus)$, or indeed on
  any additive category, called the {\em kernel-image trace}. Given a
  map $f:A\oplus U\rightarrow B\oplus U$, we say $\TrKI^U(f)\,\defined$
  iff there exist morphisms $i:A\to U$ and $k:U\to B$ such that the
  following commutes:
  \[ \xymatrix@C+5ex{
    A \ar@{-->}[r]^{i} \ar[dr]_{f_{21}} &
    U \ar[dr]^{f_{12}}\ar[d]_<>(.3){I-f_{22}}\\
    & U \ar@{-->}[r]^<>(.5){k}
    & B. \\
  }
  \]
  Whenever this condition is satisfied we define
  \begin{equation}
    \TrKI^{U}(f) = f_{11} + k\circ f_{21} = f_{11} + f_{12} \circ i : A\to B.
  \end{equation}
  To show that this is well-defined, note that $k\circ f_{21}$ does
  not depend on $i$ and $f_{12} \circ i$ does not depend on $k$, so
  $\TrKI^{U}(f)$ is independent of the choice of both $i$ and $k$.
\end{definition}

\begin{remark}
  In $\Vect$, the existence of $i$ and $k$ is equivalent to the
  following two conditions, respectively:
  \begin{eqnarray*}
    &&\im f_{21}\subseteq \im (I-f_{22}),\\
    &&\ker(I-f_{22})\subseteq \ker f_{12}.
  \end{eqnarray*}
  This explains the name ``kernel-image trace''.
\end{remark}

\begin{proposition}
  The kernel-image trace is a partial trace.
\end{proposition}

\begin{proof}
  The proof for $\Vect$ can be found in~\cite{THESIS-OCTAVIO}. Here,
  we prove it in the case of a general additive category.  Let us say
  that $(k,i)$ {\em witnesses} the existence of $\Tr(f)$ if the
  condition of Definition~\ref{def-kernel-image} holds, i.e., $f_{12}
  = k\circ(I-f_{22})$ and $f_{21} = (I-f_{22})\circ i$. In this case,
  we also write $(k,i)\witness\Tr(f)$.

  \begin{itemize}
  \item To prove naturality, assume $(k,i)\witness\Tr(f)$. Then
  $(h\circ k,i\circ g)\witness\Tr((h\oplus\id_U)\circ
  f\circ(g\oplus\id_U))$, and $\Tr((h\oplus\id_U)\circ
  f\circ(g\oplus\id_U)) = h\circ f_{11}\circ g + h\circ k\circ
  f_{21}\circ g = h\circ\Tr(f)\circ g$.
  \[ \xymatrix@C+5ex{
    A'\ar[r]^{g} & A \ar@{-->}[r]^{i} \ar[dr]_{f_{21}} &
    U \ar[dr]^{f_{12}}\ar[d]_<>(.3){I-f_{22}}\\
    && U \ar@{-->}[r]^<>(.5){k}
    & B \ar[r]^{h} & B'. \\
  }
  \]

  \item To prove dinaturality, assume $(k,i)\witness\Tr((1_{B}\oplus
  g)\circ f)$:
  \[ \xymatrix@C+5ex{
    A \ar@{-->}[rr]^{i} \ar[dr]_{f_{21}} &&
    U \ar[dr]^{f_{12}}\ar[d]_<>(.3){I-g\circ f_{22}}\\
    & U' \ar[r]^{g}& U \ar@{-->}[r]^<>(.5){k}
    & B.\\
  }
  \]
  Let $j=f_{21}+f_{22}\circ i$ and note that $(I-g\circ f_{22})\circ
  i=g\circ f_{21}$ implies $i-g\circ f_{22}\circ i=g\circ f_{21}$,
  hence $i=g\circ(f_{21}+f_{22}\circ i) = g\circ j$. Consider the
  diagram
  \[ \xymatrix@C+5ex{
    A \ar@{-->}[r]^{j} \ar[dr]_{f_{21}} &
    U'\ar[r]^{g}\ar[d]_<>(.3){I-f_{22}\circ g}&
    U \ar[dr]^{f_{12}}\ar[d]_<>(.3){I-g\circ f_{22}}\\
    & U' \ar[r]^{g}& U \ar@{-->}[r]^<>(.5){k}
    & B.\\
  }
  \]
  The left triangle commutes because $(I-f_{22}\circ g)\circ j =
  j-f_{22}\circ g\circ j = j-f_{22}\circ i = f_{21}$. The center
  square commutes because both sides are equal to $g-g\circ f_{22}\circ
  g$. Therefore $(k\circ g,j)\witness\Tr(f\circ(1_{A}\oplus g))$ and
  $\Tr(f\circ(1_{A}\oplus g)) = f_{11} + k\circ g\circ f_{21} =
  \Tr((1_{B}\oplus g)\circ f)$. This proves one direction of
  dinaturality; the opposite direction is dual.

  \item To prove strength, assume $(k,i)\witness\Tr(f)$. Then
  $(\inj_2\circ k,i\circ\pi_2)\witness\Tr(g\oplus f)$ and $\Tr(g\oplus
  f) = (g\oplus f_{11}) + \inj_2\circ k\circ f_{21}\circ\pi_2 =
  g\oplus \Tr(f)$.
  \[ \xymatrix@C+5ex{
    C\oplus A\ar[r]^{\pi_2} & A \ar@{-->}[r]^{i} \ar[dr]_{f_{21}} &
    U \ar[dr]^{f_{12}}\ar[d]_<>(.3){I-f_{22}}\\
    && U \ar@{-->}[r]^<>(.5){k}
    & B \ar[r]^{\inj_2} & B\oplus D. \\
  }
  \]

  \item To prove yanking, notice that $(1,1)\witness\Tr(\sigma_U)$, and
  $\Tr(\sigma_U) = 0 + 1 = 1$.
  \[ \xymatrix@C+5ex{
    U \ar@{-->}[r]^{1} \ar[dr]_{1} &
    U \ar[dr]^{1}\ar[d]_<>(.3){I-0}\\
    & U \ar@{-->}[r]^<>(.5){1}
    & U.\\
  }
  \]

\item To prove vanishing I, consider $f:A\oplus 0\to B\oplus 0$. Then
  $(0,0)\witness\Tr(f)$ and, writing as if the monoidal structure were
  strict, we have $\Tr(f)=f_{11}+0 = f$.
  \[ \xymatrix@C+5ex{
    A \ar@{-->}[r]^{0} \ar[dr]_{0} &
    U \ar[dr]^{0}\ar[d]_<>(.3){I-0}\\
    & U \ar@{-->}[r]^<>(.5){0}
    & B.\\
  }
  \]

  \item Finally, to prove vanishing II, consider
  \[
  f=\xmatrix{.7}{.5}{ccc}{L&M&N\\P&A&B\\Q&C&D}:A\oplus U\oplus V\to
  B\oplus U\oplus V.
  \]
  Assume $\Tr^V(f)$ is defined and witnessed by some $(k,i)$. We write
  $i=\xmatrix{.5}{.5}{cc}{E&F}$ and $k=\xmatrix{.5}{.5}{c}{G\\H}$.
  \[ \xymatrix@C+5ex@R+1ex{
    A\oplus U \ar@{-->}[r]^{\xxmatrix{.5}{.5}{cc}{~E&~F}} \ar[dr]_{\xxmatrix{.5}{.5}{cc}{~Q&~C}} &
    V \ar[dr]^{\xxmatrix{.5}{.5}{c}{~N\\~B}}\ar[d]_<>(.3){~I-D}\\
    & V \ar@{-->}[r]^<>(.3){\xxmatrix{.5}{.5}{c}{~G\\~H}}
    & B\oplus U.\\
  }
  \]
  For greater brevity, let us write $D'=I-D$ and $A'=I-A$. We
  have $HD'=B$, $GD'=N$, $D'F=C$, and $D'E=Q$. Also,
  \[ \Tr^V(f) = \xmatrix{.5}{.5}{cc}{L+GQ & M+GC\\ P+HQ & A+HC} =
  \xmatrix{.5}{.5}{cc}{L+NE & M+NF\\ P+BE & A+BF}.
  \]
  What we must show is that some $(k',i')$ witnesses $\Tr^U\Tr^V(f)$
  if and only if some $(k'',i'')$ witnesses $\Tr^{U\oplus V}(f)$, and in
  this case, $\Tr^U\Tr^V(f)=\Tr^{U\oplus V}(f)$. Let us write
  $k''=\xmatrix{.5}{.5}{cc}{R&S}$ and $i''=\xmatrix{.5}{.5}{c}{J\\K}$,
  and consider the corresponding diagrams
  \[ \xymatrix@C+7ex@R+3ex{
    A \ar@{-->}[r]^{i'} \ar[dr]_{P+BE}
    \ar@{}[drr]_<>(.4){(a)}
    \ar@{}[drr]^<>(.6){(b)}&
    U \ar[dr]^{M+NF}\ar[d]_<>(.2){I-A-BF}\\
    & U \ar@{-->}[r]^<>(.3){k'}
    & B,\\
  }
 \xymatrix@C+7ex@R+3ex{
    A \ar@{-->}[r]^{\xxmatrix{.5}{.5}{c}{~J\\~K}}
    \ar[dr]_{\xxmatrix{.5}{.5}{c}{~P\\~Q}}
    \ar@{}[drr]_<>(.4){(c)}
    \ar@{}[drr]^<>(.6){(d)}&
    U\oplus V \ar[dr]^{\xxmatrix{.5}{.5}{cc}{~M&~N}}\ar[d]_<>(.2){I-\xxmatrix{.5}{.5}{cc}{~A&~B\\~C&~D}}\\
    & U\oplus V \ar@{-->}[r]^<>(.3){\xxmatrix{.5}{.5}{cc}{~R&~S}}
    & B.\\
  }
  \]
  We note that (a) commutes iff $P+BE=A'i'-BFi'$ iff $P=A'i' -
  B(E+Fi')$, and (c) commutes iff $P=A'J-BK$ and $Q=D'K-CJ$. Now given
  $i'$ such that (a) commutes, we can set $J=i'$ and $K=E+Fi'$. Then
  $P=A'i' - B(E+Fi') = A'J-BK$, and $Q=D'E=D'(K-Fi')=D'K-D'FJ=D'K-CJ$,
  and therefore (c) commutes. Conversely, given $J$ and $K$ such that
  (c) commutes, we can set $i'=J$, and we have: $P=A'J-BK = A'i'-HD'K
  = A'i'-H(Q+CJ) = A'i'-HD'E-HD'FJ = A'i'-HD'(E+Fi') = A'i' -
  B(E+Fi')$, and therefore (a) commutes. The proof for (b) and (d) is
  dual. Finally, if both diagrams are witnessed, with $J=i'$ and
  $K=E+Fi'$, then we have $\Tr^U\Tr^V(f) = L+NE+MJ+NFJ$ and
  $\Tr^{U\oplus V}(f)=L+MJ+NK$; the two are equal because $NE+NFJ =
  N(E+Fi') = NK$.
\end{itemize}
\end{proof}

\begin{remark}
  Notice that the kernel-image partial trace generalizes the
  Haghverdi-Scott trace. Indeed, if $I-f_{22}$ is invertible, then one
  can take $i=(I-f_{22})^{-1}\circ f_{21}$ and $k=f_{12}\circ
  (I-f_{22})^{-1}$, in which case $\TrKI^U(f) = f_{11} + f_{12}
  (I-f_{22})^{-1} f_{21}$. Therefore $\TrHS^U(f)\kleeneleq\TrKI^U(f)$.
  Moreover, the kernel-image trace is strictly more general, because
  for the identity map $f=\xmatrix{.5}{.5}{cc}{1&0\\0&1}$, the
  kernel-image trace is defined but the Haghverdi-Scott trace is
  not. However, although the kernel-image trace is more defined than
  the Haghverdi-Scott trace, because of
  Proposition~\ref{prop-trace-paradox}, it still does not subsume the
  sum trace. For example, for $f=\xmatrix{.5}{.5}{cc}{0&1\\0&1}$, the
  sum trace is defined and the kernel-image trace is not.
\end{remark}

\begin{remark}
  Let $U=V\oplus W$ be a finite dimensional Hilbert space and consider
  a hermitian positive operator $A:U\to U$. Then $A$ is characterized
  by its unit ball $\ball = \s{u\in U\mid \iprod{u,Au}\leq 1}$. Let
  $\ball'\seq V$ be the orthogonal projection of $\ball$ to the
  subspace $V$. Then $\ball'$ is the unit ball of a hermitian positive
  operator $A':V\to V$, which can be explicitly defined by
  $\iprod{v,A'v}:= \min\s{\iprod{v+w,A(v+w)}\mid w\in W}$. This
  construction is intimately related to the kernel-image trace in the
  following way: If $A$ is positive, then $\TrKI^W(I-A)$ always exists
  and is equal to $I-A'$. Such a property fails to hold for the
  sum-trace (e.g., $A=\xmatrix{.5}{.5}{cc}{1&1\\1&2}$) and the
  Haghverdi-Scott trace (e.g., $A=\xmatrix{.5}{.5}{cc}{1&0\\0&0}$).
\end{remark}

\subsection{Partial trace in a symmetric monoidal subcategory of a partially traced category}
\label{ssec-trace-subcategory}

The aim of this section is to show that any symmetric monoidal subcategory of a
partially (or totally) traced category is partially traced.  Suppose
$(\cD,\otimes,I,\sigma,\Tr)$ is a partially traced category with trace
\[ \Tr^{U}_{A,B}:\cD(A\otimes U,B\otimes U)\rightharpoonup \cD(A,B).
\]
Given a symmetric monoidal subcategory $\cC\subseteq \cD$, we get a partial
trace on $\cC$, defined by $\widehat{\Tr}^U_{A,B}(f)=\Tr^U_{A,B}(f)$
if $\Tr^U_{A,B}(f)$ exists and is an element of $\cC(A,B)$, and
undefined otherwise.

Slightly more generally, we have the following:

\begin{proposition}\label{prop-subcategory}
  Let $F:\cC\rightarrow \cD$ be a faithful strong symmetric monoidal
  functor from a symmetric monoidal category $(\cC,\otimes,I,\sigma)$ to a
  partially traced category $(\cD,\otimes,I,\sigma,\Tr)$. Then we obtain a
  partial trace $\widehat{\Tr}$ on $\cC$ as follows.  For $f:A\otimes
  U\rightarrow B\otimes U$, we define $\widehat{\Tr}^U_{A,B}(f)=g$ if
  there exists some (necessarily unique) $g:A\rightarrow B$ such that
  $F(g) = \Tr^{FU}_{FA,FB}(m^{-1}_{B,U}\circ F(f)\circ m_{A,U})$ is
  defined, and $\widehat{\Tr}^U_{A,B}(f)$ undefined otherwise.
\end{proposition}

\begin{proof}
The details can be found in~\cite{THESIS-OCTAVIO}.
\end{proof}

\begin{remark}
  This yields a large class of examples of partially traced categories
  that are related to known totally traced categories. For example,
  consider the category $\SRelfin$ of finite sets and stochastic
  maps. Here, a stochastic map from $A$ to $B$ is a function from $A$
  to sub-probability distributions on $B$, with the obvious identities and
  composition. In elementary terms, this is a $[0,1]$-valued matrix
  whose columns have sum $\leq 1$. With the tensor $\oplus$ (disjoint
  union), this category is totally traced. With the tensor $\otimes$
  (cartesian product), it is not totally traced; however,
  $(\SRelfin,\otimes)$ can be regarded as a symmetric monoidal
  subcategory of the totally traced category $(\Vectfin,\otimes)$ of
  finite dimensional real vector spaces and linear functions.
  Therefore it inherits a partial trace.

  Other examples of partial traces arise in this way from the models
  for quantum computing considered in {\cite{Sel2004}}, for example on
  completely positive maps and on superoperators. Such examples are
  described in detail in {\cite{THESIS-OCTAVIO}}.
\end{remark}

\section{Paracategories and their completion}
\label{sec-paracat}

The goal of the remainder of this paper is to prove a strong converse
to Proposition~\ref{prop-subcategory}, i.e.: every partially
traced category arises as a symmetric monoidal subcategory of a
totally traced category. More precisely, we show that every partially
traced category can be faithfully embedded in a compact closed
category in such a way that the trace is preserved and reflected.

Our construction uses a partial version of the $\Int$-construction of
Joyal, Street, and Verity~\cite{JSV96}. When we try to apply the
$\Int$-construction to a partially traced category $\cC$, we find that
composition in $\Int(\cC)$ is in general only partially defined. We
therefore consider a notion of ``categories'' with partially defined
composition, namely, Freyd's paracategories~\cite{HerMat03}.
Specifically, we introduce the notion of a strict symmetric compact
closed paracategory.

We first show in Section~\ref{sec-paracat} that every partially traced
category can be fully and faithfully embedded in a compact closed
paracategory, by an analogue of the $\Int$-construction.  We then show
in Section~\ref{sec-pint} that every compact closed paracategory can
be embedded (faithfully, but not necessarily fully) in a compact
closed (total) category, using a construction similar to that of
Freyd. Finally, every compact closed category is (totally) traced,
yielding the desired result in Section~\ref{sec-representation}.

\subsection{Paracategories}

We recall Freyd's notion of paracategory.  A reference on this subject
is~\cite{HerMat03}. Informally, a paracategory is a category with
partially defined composition.

\begin{definition}
\rm A \textit{(directed) graph} $\cC$ consists of:
\begin{itemize}
\item a class of {\em objects} $\obj (\cC)$, and
\item for every pair of objects $A,B$, a set $\cC(A,B)$ of
  {\em arrows} from $A$ to $B$.
\end{itemize}
If $\cC,\cD$ are graphs, a {\em graph homomorphism} $F:\cC\to\cD$ is
given by a (total) function $F:\obj(\cC)\rightarrow \obj(\cD)$ and a
family of (total) functions $F:\cC(A,B)\rightarrow\cD(FA,FB)$. We say
that $F$ is {\em faithful} if $F:\cC(A,B)\rightarrow\cD(FA,FB)$ is
one-to-one for all $A,B$.
\end{definition}

\begin{definition}
  \rm Let $\cC$ be a graph. We define $\PC$, the {\em path
    category} of $\cC$, by $\obj(\PC)=\obj(\cC)$ and
  arrows from $A_0$ to $A_n$ are finite sequences
  $(A_0,f_1,A_1,f_2,\ldots,f_n,A_n)$ of alternating objects and arrows
  of the graph $\cC$, where $n\geq 0$ and $f_i:A_{i-1}\to A_i$
  for all $i$. We say that $n$ is the {\em length} of the path. To be
  clear, equality of arrows is literal equality of sequences.
  Composition is defined by concatenation, and the identity arrow at
  an object $A$ is the path $(A)$ of length zero.
\end{definition}

\begin{notation}
  For the sake of simplicity, we often write
  $\vec{f}=f_1,f_2,\ldots,f_n$ for a path, when the objects are
  understood. We use the comma ``,'' for concatenation.  We also
  write $\epsilon_A=(A)$ for the path of length zero at $A$, so
  that $\epsilon_A,\vec{f}=\vec{f}=\vec{f},\epsilon_B$ for a path
  $\vec{f}: A \rightarrow B$.
\end{notation}

Recall the definition of Kleene equality ``$\kleeneeq$'' and directed
Kleene equality ``$\kleeneleq$'' from Definition~\ref{def-kleene-equality}.

\begin{definition}\label{def-paracat}
  \rm A \textit{paracategory} $(\cC,[-])$ consists of a directed graph
  $\cC$ and a family of partial operations
  $[-]_{A,B}:\PC(A,B)\rightharpoonup\cC(A,B)$, called {\em (partial)
    composition}, which satisfies the following axioms. We usually
  omit the subscripts.
\begin{itemize}
\item[(a)] for all $A$, $[\epsilon_A]\defined$, i.e.,
$[-]$ is a total operation on empty paths;
\item[(b)] for paths of length one, $[f]\defined$ and $[f]=f$ (or
  equivalently, using Kleene equality, $[f]\kleeneeq f$);
\item[(c)] for all paths $\vec{r}:A\to B$, $\vec{f}:B\to C$, and
  $\vec{s}:C\to D$,
  \[ [\vec{f}\,]\,\defined \quad\mbox{implies}\quad
  [\vec{r},[\vec{f}\,],\vec{s}\,]\kleeneeq[\vec{r},\vec{f},\vec{s}\,].
  \]
\end{itemize}
\end{definition}

\begin{remark}
  Every category $\cC$ can be regarded as a paracategory with
  $[f_1,\ldots,f_n] = f_n \circ \ldots \circ f_1$. In this case,
  composition is a totally defined operation.
\end{remark}

\begin{remark}[Identity]\label{rem-identities}
  In any paracategory, we will write $1_A=[\epsilon_A]$. Note that by
  (a) and (c), it follows that $[\vec r, 1_A, \vec s\,] \kleeneeq [\vec
  r,\vec s\,]$ for all $\vec r$, $\vec s$, so the arrow $1_A$ indeed
  behaves like an identity.
\end{remark}

\begin{remark}[Inverses]\label{rem-inverse}
  If there are two arrows $b:A\to B$ and $b^{-1}:B\to A$ in a
  paracategory such that $[b,b^{-1}]=1_A$ and $[b^{-1},b]=1_B$, then
  for all arrows $f:X\to A$ and $g:X\to B$, $[f,b]=g$ iff
  $f=[g,b^{-1}]$. Namely, from the assumption $[f,b]= g$, we
  can deduce $[g,b^{-1}]\kleeneeq [[f,b],b^{-1}]\kleeneeq
  [f,b,b^{-1}]\kleeneeq[f,[b,b^{-1}]]\kleeneeq [f,1]\kleeneeq
  [f]=f$, and the proof of the converse is similar.
\end{remark}

\begin{convention}\label{con-f}
  We extend any graph homomorphism $F:\cC\to\cD$ to paths by the
  following slight abuse of notation: for any path $\vec f =
  f_1,\ldots,f_n$, we write
  \[ F\vec f := Ff_1,\ldots,Ff_n.
  \]
\end{convention}

\begin{definition}\label{def-functor-of-paracat}
  \rm Let $(\cC,[-])$ and $(\cD,[-]')$ be paracategories. A
  \textit{functor} of paracategories is a graph homomorphism
  $F:\cC\to\cD$ such that for all $\vec{p}$,
  \[   F[\,\vec{p}\,] \kleeneleq [\,F\vec{p}\,]'.
  \]
\end{definition}

We note that functors of paracategories preserve identities. Indeed,
since $[\epsilon_A]\defined$, we have $F(1_A) = F[\epsilon_A] =
[F\epsilon_A] = [\epsilon_{FA}] = 1_{FA}$.

\begin{definition}
  Let $(\cC,[-])$ and $(\cD,[-]')$ be paracategories. Then the
  paracategory $\cC\times\cD$ has $\obj(\cC\times\cD) :=
  \obj(\cC)\times\obj(\cD)$ and $(\cC\times\cD)((A,A'),(B,B')) :=
  \cC(A,B)\times \cD(A',B')$, and $[(f_1,g_1),\ldots,(f_n,g_n)]
  :\kleeneeq ([f_1,\ldots,f_n], [g_1,\ldots,g_n])$. Then
  $\cC\times\cD$ is a categorical product in the category of
  paracategories and functors.
\end{definition}

\subsection{Symmetric monoidal paracategories}

\begin{definition}\label{def-ssmpc}
  \rm A \textit{strict symmetric monoidal paracategory}
  $(\cC,[-],\otimes,I,\sigma)$ consists of:
  \begin{enumerate}
  \item[(a)] a paracategory $(\cC,[-])$;
  \item[(b)] a functor of paracategories $\otimes:\cC\times\cC\to\cC$, and
    an object $I$, satisfying
    \begin{itemize}
    \item $(A\otimes B)\otimes C=A\otimes (B\otimes C)$ on objects and
      $(f\otimes g)\otimes h=f\otimes (g\otimes h)$ on arrows
      (associativity);
    \item $A\otimes I=A=I\otimes A$ on objects and $f\otimes
      1_I=f=1_I\otimes f$ on arrows (unit laws);
    \end{itemize}
  \item[(c)] for all objects $A$ and $B$, an arrow
    $\sigma_{A,B}:A\otimes B\rightarrow B\otimes A$ such that:
    \begin{itemize}
    \item[-] for every $f:X\otimes B\otimes A\rightarrow Y$, $g:Y\rightarrow
      X\otimes A\otimes B$, $[1_X\x \sigma_{A,B},f]\defined$ and
      $[g,1_X\x \sigma_{A,B}]\defined$ (totality);
    \item[-] for every $f:A\rightarrow A'$ and $g:B\rightarrow B'$:
      $[f\otimes 1_B,\sigma]=[\sigma,1_B\otimes f]$ and $[1_A\otimes
      g,\sigma]=[\sigma,g\otimes 1_A]$ (naturality);
    \item[-] for every $A$ and $B$:
      $[\sigma_{A,B},\sigma_{B,A}]=1_{A\otimes B}$ (symmetry);
    \item[-] for every $A$, $B$, and $C$: $\sigma_{A,B\x C} =
      [\sigma_{A,B}\x 1_C, 1_B\x\sigma_{A,C}]$ (``hexagon'' axiom).
    \end{itemize}
  \end{enumerate}
\end{definition}

The assumption that $\x$ is a functor of paracategories explicitly
means that it is a graph homomorphism satisfying
\begin{equation}\label{eqn-ssmpc}
  [f_1,\dots,f_n]\otimes [g_1\dots,g_n]\kleeneleq [f_1\otimes g_1,\dots,f_n\otimes g_n].
\end{equation}

\begin{lemma}\label{lem-PARACAT2}
  For arrows $p:A\to B$, $q:C\to D$ of a strict symmetric monoidal
  paracategory $\cC$, we have that $[p\otimes 1_C,1_B\otimes q]$ and
  $[1_A\otimes q,p\otimes 1_D]$ are both defined and equal to
  $p\otimes q$. Moreover, for any paths $\vec f$ and $\vec g$, we have 
  $[\vec f, p\otimes 1,1\otimes q, \vec g\,]
  \kleeneeq [\vec f, p\otimes q, \vec g\,]
  \kleeneeq [\vec f, 1\otimes q,p\otimes 1, \vec g\,]$.
\end{lemma}

\begin{proof}
  Let $p:A\rightarrow B$ and $q:C\rightarrow D$. By
  Remark~\ref{rem-identities} and functoriality, $p\otimes
  q=[p,1_B]\otimes [1_C,q]\kleeneleq [p\otimes 1_C,1_B\otimes q]$ and
  $p\otimes q=[1_A,p]\otimes [q,1_D]\kleeneleq[1_A\otimes q,p\otimes
  1_D]$. But $p\otimes q$ is totally defined, so all of the above
  terms are defined and equal.  Using this and axiom (c) of
  paracategories, we have for any paths $\vec f$ and $\vec g$,
  \[ [\vec f, p\otimes 1,1\otimes q, \vec g\,]
  \kleeneeq [\vec f, [p\otimes 1,1\otimes q\,], \vec g\,]
  \kleeneeq [\vec f, p\otimes q, \vec g\,],
  \]
  and similarly for $[\vec f, 1\otimes q,p\otimes 1, \vec g\,]$.
\end{proof}

\begin{lemma}\label{lem-ssmpc}
  In the definition of a strict symmetric monoidal paracategory,
  condition (\ref{eqn-ssmpc}) is equivalent to the following pair of
  conditions:\rm
  \begin{itemize}
  \item[(a)] $[f,f']\otimes [g,g']\kleeneleq [f\otimes g,f'\otimes
    g']$ where $f,g,f',g'$ are arrows of $\cC$; and
  \item[(b)] $1\otimes [\, \vec{p}\, ]\kleeneleq [\, 1\otimes
    \vec{p}\, ]$ and $[ \,\vec{p}\, ]\otimes 1 \kleeneleq [\,
    \vec{p}\otimes 1 \,]$.
  \end{itemize}
  Note that in stating (b), we have used Convention~\ref{con-f}, so by
  definition, $1\x\vec p = 1\x p_1,\ldots,1\x p_n$.
\end{lemma}

\begin{proof}
  Clearly (\ref{eqn-ssmpc}) implies (a). Also, by
  Remark~\ref{rem-identities} and (\ref{eqn-ssmpc}), we have
  $1\x[\,\vec p\,]\kleeneeq [1,\ldots,1]\x[\,\vec p\,]\kleeneleq
  [1\x\vec p\,]$, and similarly $[ \,\vec{p}\, ]\otimes 1 \kleeneleq
  [\, \vec{p}\otimes 1 \,]$, so (\ref{eqn-ssmpc}) implies (b).  For
  the converse, first note that the proof of Lemma~\ref{lem-PARACAT2}
  only uses property (a). Assume $[\vec f\,]$ and $[\vec g\,]$ are
  defined. Then by Lemma~\ref{lem-PARACAT2} and (b), we have $[\vec
  f\,]\x[\vec g\,] = [[\vec f\,]\x 1,1\x [\vec g\,]] \kleeneleq [\vec f\x
  1,1\x\vec g\,]\kleeneeq\ldots\kleeneeq[f_1\x 1,1\x g_1,\ldots,f_n\x
  1,1\x g_n]\kleeneeq[f_1\x g_1,\ldots,f_n\x g_n]$.
\end{proof}

\begin{definition}
  \rm Let $(\cC,[-],\otimes,I,\sigma)$ and
  $(\cD,[-]',\otimes',I',\sigma')$ be strict symmetric monoidal
  paracategories. A functor between them is \textit{strict symmetric
    monoidal} when $F(A)\otimes' F(B)=F(A\otimes B)$ and $F(I)=I'$ on
  objects, and $F(f)\otimes' F(g)=F(f\otimes g)$ and
  $F(\sigma)=\sigma'$ on arrows.
\end{definition}

\subsection{The completion of symmetric monoidal paracategories}

In this section, we will prove that every strict symmetric monoidal
paracategory can be faithfully embedded in a strict symmetric monoidal
category. From now on, $\cC$ denotes a strict symmetric monoidal
paracategory.

\begin{definition}\label{def-PARACAT3}
\rm A {\em congruence relation} $\cS$ on $\PC$ is given as
 follows: for every pair of objects $A,B$, an equivalence relation
   $\sim_{\cS}^{A,B}$ on the hom-set $\PC(A,B)$, satisfying the following
   axioms. We usually omit the superscripts when they are clear from
   the context.

\begin{itemize}
\item[(1)] If $\vec p \sim_{\cS} \vecp{p}$ and $\vec q \sim_{\cS} \vecp{q}$, then
    $\vec p,\vec q \sim_{\cS} \vecp{p},\vecp{q}$.
\item[(2)] Whenever $[\vec{p}\,]\defined$, then $\vec{p} \sim_{\cS} [\vec{p}\,]$.
\item[(3)] If $\vec p \sim_{\cS} \vec q$, then $\vec p\otimes 1 \sim_{\cS} \vec
     q\otimes 1$ and $1\otimes \vec p \sim_{\cS} 1\otimes \vec q$.
\end{itemize}
\end {definition}

\begin{definition}
 We define a particular congruence relation $\hS$ as follows: $\vec{p} \sim_{\hS} \vec{q}$ if and only if for all objects $A,B$ and all $\vec{r},\vec{s}$,
 \[ [\,\vec{r},1_A\otimes \vec{p}\otimes 1_B,\vec{s}\,]
 \kleeneeq[\,\vec{r},1_A\otimes \vec{q}\otimes 1_B,\vec{s}\,].
 \]
\end{definition}

\begin{remark}\label{rem-PARACAT1}
  \rm It should be observed that $\vec{p}\sim_{\hS}\vec{q}$
  implies $[\,\vec{p}\,]\kleeneeq[\,\vec{q}\,]$ by letting $A = B = I$
  and $\vec{r}, \vec{s}$ be empty lists.
\end{remark}

\begin{lemma}
$\hS$ is a congruence relation.
\end{lemma}

\begin{proof}
  We need to show axioms (1)--(3). To show (1), note that
  $\vec{p}\sim_{\hS} \vecp{p}$ and $\vec{q}\sim_{\hS}
  \vecp{q}$ implies
  \begin{eqnarray*}
    [\,\vec{r},1_A\otimes (\vec{p},\vec{q})\otimes 1_B,\vec{s}\,]
    &\kleeneeq&
    [\,\vec{r},1_A\otimes \vec{p}\otimes 1_B,
    1_A\otimes \vec{q}\otimes 1_B,\vec{s}\,]
    \\ &\kleeneeq& [\,\vec{r},1_A\otimes \vecp{p}\otimes 1_B,
    1_A\otimes \vec{q}\otimes 1_B,\vec{s}\,]
    \\ &\kleeneeq& [\,\vec{r},1_A\otimes \vecp{p}\otimes 1_B,
    1_A\otimes \vecp{q}\otimes 1_B,\vec{s}\,]
    \\ &\kleeneeq&
    [\,\vec{r},1_A\otimes (\vecp{p},\vecp{q})\otimes 1_B,\vec{s}\,],
  \end{eqnarray*}
  where the first and last equation is just the definition of
  $\otimes$ on paths. Therefore $\vec p,\vec q \sim_{\cS}
  \vecp{p},\vecp{q}$.  To show (2), assume $[\vec p\,]\defined$. Then by
  Lemma~\ref{lem-ssmpc}(b), $1_A\x[\vec{p}\,]\x1_B =
  [1_A\x\vec{p}\x1_B]$ is defined, and from the laws of
  paracategories, 
  \[ [\,\vec{r},1_A\x\vec{p}\x1_B,\vec{s}\,] \kleeneeq
  [\,\vec{r},[1_A\x\vec{p}\x1_B],\vec{s}\,] \kleeneeq
  [\,\vec{r},1_A\x[\vec{p}\,]\x1_B,\vec{s}\,].
  \]
  Property (3) is immediate from the definition of $\hS$.
\end{proof}

\begin{definition}\label{def-sim}
\rm
 Let $\sim$ be the smallest congruence relation on $\PC$,
i.e., the intersection of all congruence relations.
\end{definition}

\begin{lemma}\label{lem-PARACAT6}
  $\vec{p}\sim\vec{q}$ implies $[\,\vec{p}\,]\kleeneeq[\,\vec{q}\,]$.
\end{lemma}

\begin{proof}
  Since $\sim$ is the smallest congruence relation,
  $\vec{p}\sim\vec{q}$ implies $\vec{p}\sim_{\hS}\vec{q}$, which
  implies $[\,\vec{p}\,]\kleeneeq[\,\vec{q}\,]$ by
  Remark~\ref{rem-PARACAT1}.
\end{proof}

\begin{corollary}\label{cor-path-singleton}
  If $\vec p, q:A\to B$ are paths where $q$ is of length 1, then
  $\vec p\sim q$ iff $[\vec p\,]\defined$ and $[\vec p\,]=q$.
\end{corollary}

\begin{proof}
  The left-to-right direction is obvious from
  Lemma~\ref{lem-PARACAT6} and axiom (b) of paracategories. The
  right-to-left direction is Definition~\ref{def-PARACAT3}(2).
\end{proof}

\begin{corollary}\label{cor-PARACAT6b}
  If $p,q:A\to B$ are paths of length 1, then $p\sim q$ iff $p=q$.
\end{corollary}

\begin{proof}
  From Lemma~\ref{lem-PARACAT6} and axiom (b) of paracategories.
\end{proof}

\begin{definition}\label{def-pathtensor}
We now introduce the following notation, where $\vec{f}$ and $\vec{g}$
are paths, not necessarily of the same length.
\begin{equation}
  \vec{f}\otimesp\vec{g} := (\vec{f}\otimes 1),\,(1\otimes\vec{g}).
\end{equation}
\end{definition}

Note that, as a path, this is not equal to
$(1\otimes \vec{g}),(\vec{f}\otimes 1)$. However, we will show that
they are congruent.

\begin{lemma}\label{lem-PARACAT5}
  Let $\cS$ be a congruence relation on $\PC$. Then
  $\vec{f}\sim_{\cS}\vecp{f}$ and $\vec{g}\sim_{\cS}\vecp{g}$ implies
  $\vec{f}\otimesp\vec{g} \sim_{\cS}\vecp{f}\otimesp\vecp{g}$.
\end{lemma}

\begin{proof}
  Assuming $\vec{f}\sim_{\cS}\vecp{f}$ and
  $\vec{g}\sim_{\cS}\vecp{g}$, we immediately have $(\vec{f}\otimes
  1),(1\otimes\vec{g})\sim_{\cS}(\vecp{f}\otimes
  1),(1\otimes\vecp{g})$ by Definition~\ref{def-PARACAT3}(1) and (3).
\end{proof}

\begin{lemma}
  Let $\cS$ be a congruence relation of $\PC$. Then
  \[ (\vec{f}\otimes 1),(1\otimes \vec{g})\sim_{\cS}
  (1\otimes\vec{g}),(\vec{f}\otimes 1).
  \]
\end{lemma}

\begin{proof}
  First, consider arrows $f,g$ of $\cC$. By
  Lemma~\ref{lem-PARACAT2}, we have $[f\x 1,1\x g] = f\x g = [1\x
  g,f\x 1]$, and in particular, these terms are defined. Therefore by
  Definition~\ref{def-PARACAT3}(2),
  \[ f\x 1,1\x g \sim_{\cS} [f\x 1,1\x g] = [1\x g,f\x 1] \sim_{\cS}
  1\x g,f\x 1.
  \]
  The general claim follows by induction, using
  Definition~\ref{def-PARACAT3}(1) and transitivity.
\end{proof}

\begin{proposition}\label{prop-quotient-ssmpc}
  Let $\cC$ be a strict symmetric monoidal paracategory, and let $\cS$
  be a congruence relation on $\PC$.  Then the quotient $\PC/\cS$ is a
  strict symmetric monoidal category.
\end{proposition}

\begin{proof}
  $\PC/\cS$ is evidently a category; its objects are those of $\cC$
  and its morphisms $\bar{\vec{f}}=\overline{f_1,\ldots,f_n}$ are
  $\cS$-equivalence classes of paths. Composition is given by
  concatenation of paths, and is well-defined by
  Definition~\ref{def-PARACAT3}(1).  A bifunctor $\otimespp:\PC/\cS \times
  \PC/\cS\rightarrow \PC/\cS$ is defined by
  $\bar{\vec{f}}\otimespp\bar{\vec{g}} =
  \overline{\vec{f}\otimesp\vec{g}}$, and is
  well-defined by Lemma~\ref{lem-PARACAT5}. The symmetry is given by
  $\overline{\sigma_{A,B}}:A\x B\to B\x A$. The laws of strict symmetric
  monoidal categories are easily verified.
\end{proof}

From now on, we also write ``$\semi$'' to denote composition in the quotient
category written in diagrammatic order, i.e., concatenation of
(equivalence classes of) paths. Also, by a slight abuse of notation,
we write $1_A = \overline{1_A}$ for the identities in $\PC/\cS$, i.e.,
this is the equivalence class of the empty path at $A$.

We are now ready to prove that every strict symmetric monoidal
paracategory can be faithfully embedded in a strict symmetric monoidal
category.

\begin{definition}\label{def-f}
  If $\cC$ is a strict symmetric monoidal paracategory, $\cS$ a
  congruence, and $\PC/\cS$ is the quotient category, we define a
  functor of paracategories $F:\cC\rightarrow \PC/\cS$, where the
  category $\PC/\cS$ is understood as a (total) paracategory, as
  follows.
  \begin{itemize}
  \item[-] on objects, $F$ is the identity, and
  \item[-] on arrows, $F(f)=\overline{f}$, the equivalence class of a
    path of length 1.
  \end{itemize}
\end{definition}

\begin{proposition}\label{prop-PARACAT7}
  $F:\cC\rightarrow \PC/\cS$ is a well-defined functor of symmetric
  monoidal paracategories. Moreover, if $\cS$ is the smallest
  congruence relation $\sim$, then $F$ is faithful. 
\end{proposition}

\begin{proof}
  Observe that $F$ is indeed a functor of paracategories as in
  Definition~\ref{def-functor-of-paracat}: when $[\vec{f}\,]$ is
  defined, then by Definition~\ref{def-PARACAT3}(2) $[\vec{f}\,]\sim_{\cS}
  \vec{f}$, hence
  \[ F[\vec{f}\,] = \overline{[\vec{f}\,]} = \overline{\vec{f}} =
  \overline{f_1,\ldots,f_n} = \overline{f_1}\semi\ldots\semi\overline{f_n} =
  Ff_1\semi\ldots\semi Ff_n.
  \]
  Moreover, $F$ is strictly monoidal: by Lemma~\ref{lem-PARACAT2},
  Definition~\ref{def-PARACAT3}(2), Definition~\ref{def-pathtensor} and by
  definition of the tensor on $\PC$, we have
  \[ F(f\otimes g) = \overline{f\otimes g} = \overline{[f\otimes
    1_B,1_C\otimes g]} = \overline{f\otimes 1_B,1_C\otimes g} =
  \overline{f\otimesp g} = Ff\otimespp Fg.
  \]
  Also, trivially, $F(\sigma)=\overline\sigma$.

  For general $\cS$, the functor $F$ may not be faithful. For a
  trivial example, consider the maximal relation $\cS=\PC\times\PC$,
  which is always a congruence. However, if $\cS$ is the smallest
  congruence relation, then $F$ is faithful by
  Corollary~\ref{cor-PARACAT6b}. Indeed, by Remark~\ref{rem-PARACAT1},
  this is true for any congruence relation satisfying $\cS\subseteq
  \hS$.
\end{proof}

\subsection{Compact closed paracategories}

\begin{definition}
  \rm A \textit{(strict symmetric) compact closed paracategory}
  $(\cC,[-],\otimes,I,\sigma,\eta,\eps)$ is a strict symmetric
  monoidal paracategory, equipped for every object $A$ with a given
  object $A^*$ and given arrows $\eta_A:I\rightarrow A^*\otimes A$,
  $\eps_A :A\otimes A^*\rightarrow I$, such that
  \begin{itemize}
  \item $[1_A\x\eta_C, f\x 1_C]$, $[g\x 1_{C^*}, 1_B\x\eps_C]$,
    $[\eta_A\x 1_B,1_{A^*}\x h]$, and $[1_A\x k, \eps_A\x 1_C]$ are
    defined, for all $f:A\x C^*\to B$, $g:A\to B\x C$, $h:A\x B\to
    C$, and $k:B\to A^*\x C$ (totality);
  \item $[1_A\x \eta_A,\eps_A\x 1_A ]=1_A$ and $[\eta_A\x 1_{A^*},
    1_{A^*}\x \eps_A]=1_{A^*}$.
  \end{itemize}
\end{definition}

\begin{theorem}\label{thm-faithful-embed-comp-closed-para}
  If $\cC$ is a compact closed paracategory, then $\PC/\cS$ is a
  compact closed category. In particular, every compact closed
  paracategory can be faithfully embedded in a compact closed
  category.
\end{theorem}

\begin{proof}
  We must show that $\PC/\cS$, with $\eta'=\overline{\eta}$ and
  $\eps'=\overline{\eps}$, is compact closed. This is easily
  verified. For example, the condition $[1\x\eta, \eps\x 1]\defined$
  implies:
  \begin{eqnarray*}
    1_A \otimespp \overline{\eta} \semi \overline{\eps} \otimespp 1_A
    &=& \overline{1_A\otimes\eta }\semi\overline{\eps\otimes 1_A}\\
    &=& \overline{1_A\otimes\eta ,\eps\otimes 1_A}\\
    &=& \overline{[1_A\otimes\eta ,\eps\otimes 1_A]}\\
    &=& \overline{1_A} = 1_A.
  \end{eqnarray*}
  The proof of $\overline{\eta}\otimespp 1_{A^*}\semi 1_{A^*}
  \otimespp\overline{\eps} = 1_{A^*}$ is similar.
\end{proof}

\begin{remark}\label{rem-trace-cc-paracat}
  By analogy with Proposition~\ref{prop-canonical-trace}, in any
  compact closed paracategory, we can define the {\em trace} of an
  arrow $f:A\x U\to B\x U$ to be
  \[ \Tr^U_{A,B}(f) ~\kleeneeq~ [\id_A\x\eta_U, \id_A\x\sigma_{U^*,U},
  f\x\id_{U^*}, \id_B\x\eps_U] : A \to B.
  \]
  Then $\Tr^U_{A,B}$ is of course a partially defined operation.
\end{remark}

Recall from Definition~\ref{def-sim} that $\sim$ is the smallest
congruence relation on $\PC$.

\begin{theorem}\label{thm-cc-trace}
  The functor $F:\cC\to\PC/{\sim}$ preserves and reflects the trace.
  This means that for all $f:A\x U\to B\x U$ and $g:A\to B$ in $\cC$,
  we have $\Tr^U(f)=g$ iff ${\Tr^{FU}}F(f)=F(g)$.
\end{theorem}

\begin{proof}
  By definition, we have ${\Tr^{FU}}F(f)=F(g)$ in $\PC/{\sim}$ if and
  only if $\id_A\x\eta_U, \id_A\x\sigma_{U^*,U}, f\x\id_{U^*},
  \id_B\x\eps_U\sim g$ is an equivalence of paths in $\PC$. By
  Corollary~\ref{cor-path-singleton}, this is the case iff
  $[\id_A\x\eta_U, \id_A\x\sigma_{U^*,U}, f\x\id_{U^*},
  \id_B\x\eps_U]=g$ in $\cC$, i.e., $\Tr^U(f)=g$.
\end{proof}

\subsection{The universal property of $\PC/{\sim}$}

We can strengthen Proposition~\ref{prop-PARACAT7} by noting that the
faithful embedding satisfies a universal property when $\cS$ is the
smallest congruence relation.

\begin{theorem}\label{thm-freeness}
  Let $\cC$ be a strict symmetric monoidal paracategory, and let
  $\sim$ be the smallest congruence relation on $\PC$. Then the
  category $\PC/{\sim}$ satisfies the following property: for any
  strict symmetric monoidal category $\cD$ and any strict symmetric
  monoidal functor $G:\cC\rightarrow \cD$ of paracategories, there
  exists a unique strict symmetric monoidal functor
  $L:\PC/{\sim}\rightarrow \cD$ such that $L\circ F=G$, where $F$ is
  the canonical functor as in Definition~\ref{def-f}.
  \[\xymatrix@=25pt{
    \cC\ar[rr]^{F}\ar[rrd]_{G}
    && \PC/{\sim} \ar[d]^{L}\\
    && \cD }
  \]
\end{theorem}

\begin{proof}
  For consistency of notation, let us write ``$\semi$'' for
  composition in $\cD$ in diagrammatic order. Define a family of
  relations $\cS$ on $\PC$ by:
  \[ \vec{f} \sim_{\cS} \vec{g}
  \quad :\Longleftrightarrow \quad
  G(f_1)\semi\dots\semi
  G(f_n)=G(g_1)\semi\dots\semi G(g_m),
  \]
  where $\vec{f}=f_1,\ldots,f_n$ and $\vec{g}=g_1,\ldots,g_m.$ We
  claim that $\cS$ is a congruence relation. Clearly, it is an
  equivalence relation. Properties (1) and (3) of
  Definition~\ref{def-PARACAT3} are trivialities; for (2), note that when
  $[\vec f\,]\defined$, then $G[\vec f\,] = Gf_1\semi\ldots\semi Gf_n$ by
  Definition~\ref{def-functor-of-paracat}, hence $[\vec f\,]\sim_{\cS} \vec f$.

  We define $L$ as follows:
  \begin{center}
    $L(A)=G(A)$ on objects and
    $L(\bar{\vec{p}}\,)=G(p_1)\semi\dots\semi G(p_n) $, where
    $\vec{p}=p_1,\dots , p_n$.
  \end{center}
  $L$ is well-defined because $\vec{p}\sim\vec{q}$ implies
  $\vec{p}\sim_{\cS}\vec{q}$, and this implies
  $L(\bar{\vec{p}}\,)=L(\bar{\vec{q}}\,)$.  $L$ is easily seen to be a
  strict symmetric monoidal functor satisfying $L\circ F=G$.

  For uniqueness, consider any other such functor $L'$. Then $L'(A) =
  L'(FA) = GA = LA$ and
  $L'(\bar{\vec{p}}\,) = L'(\bar p_1\semi\ldots\semi\bar p_n)
  = L'(Fp_1,\ldots,Fp_n)
  = L'(Fp_1)\semi\ldots\semi L'(Fp_n)
  = G(p_1)\semi\ldots\semi G(p_n) = L(\bar{\vec{p}}\,)$, so $L'=L$.
\end{proof}

An analogous result holds with respect to compact closed
paracategories and compact closed categories.

\section{The $\Int$-construction for partially traced categories}
\label{sec-pint}

Joyal, Street, and Verity proved in \cite{JSV96} that every (totally)
traced monoidal category $\cC$ can be faithfully embedded in a compact
closed category $\Int(\cC)$. Here we show, by a similar construction,
that every {\em partially} traced category $\cC$ can be faithfully
embedded in a compact closed {\em paracategory} $\IntpC$. We call
the corresponding construction the {\em partial $\Int$-construction}.
We assume without loss of generality that $\cC$ is strictly monoidal.

\subsection{The definition of $\IntpC$}

\begin{definition}
  \rm To any partially traced symmetric strictly monoidal category
  $\cC$, we associate a graph $\IntpC$ as follows.
  \begin{itemize}
  \item an object is a pair $(\Ap,\Am)$ of objects of the category $\cC$.
  \item an arrow $f:(\Ap_0,\Am_0)\rightarrow (\Ap_1,\Am_1)$ is an
    arrow $f:\Ap_0\otimes \Am_1\rightarrow \Ap_1\otimes \Am_0$ in the
    category $\cC$.
  \end{itemize}
\end{definition}

To make $\IntpC$ into a paracategory, we need to define a partial
composition operation $[-]$ on paths. Before giving the formal
definition, we first illustrate the idea in the case of a path $\vec p
= p_1,p_2,p_3$ of length 3, where
\[ \pmobj{0}\catarrow{p_1}
\pmobj{1}\catarrow{p_2}
\pmobj{2}\catarrow{p_3}
\pmobj{3}.
\]
In this case, the partial composition $[\vec p\,]:\pmobj{0}\to\pmobj{3}$
is defined as follows:
\begin{equation}\label{eqn-partial-comp-len-3}
  [\vec p\,]:\kleeneeq
  \mp{0.25}{\includegraphics[scale=.75]{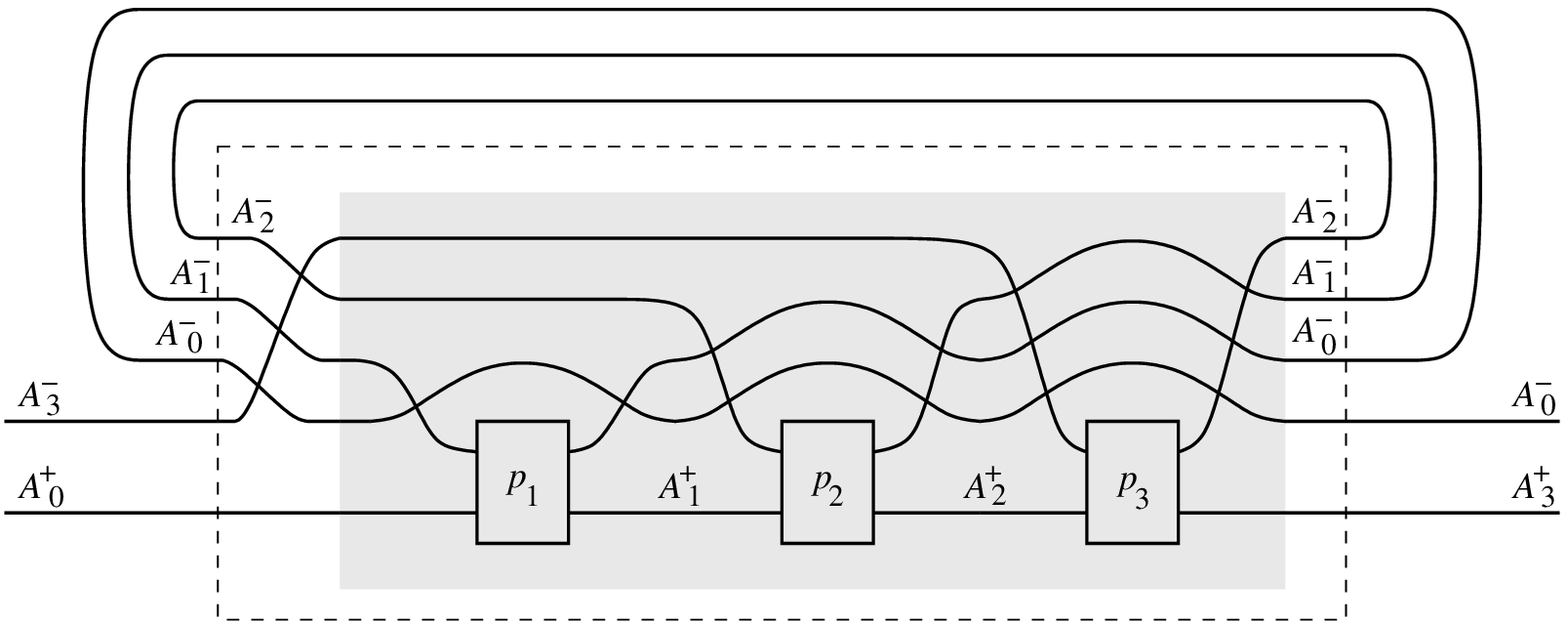}}
\end{equation}
See Section~\ref{ssec-graphical-partial} for our conventions regarding
the graphical language. In particular, the trace shown is a single
trace over the object $A_0^-\x A_1^-\x A_2^-$. Note that this trace may
be undefined, and therefore $[\vec p\,]$ is a partial operation.

To define $[\vec p\,]$ for paths of arbitrary length, we give a
recursive definition. We first recursively define an auxiliary
operation, corresponding to the contents of the shaded area in
(\ref{eqn-partial-comp-len-3}).

\begin{definition}
  We define an auxiliary (total) operation $\sem{-}$, called {\em
    precomposition}. This operation assigns to each path $\vec
  p=p_1,\ldots,p_n:\pmobj{0}\to\pmobj{n}$, with $n\geq 0$ and
  $p_i:\pmobj{i-1}\to\pmobj{i}$, a morphism
  \[ \sem{\vec p} : \Ap_0\x \vecc \Am\x \Am_n \to \Ap_n\x \Am_0\x\vecc\Am,
  \]
  where $\vecc\Am = \Am_0\x\ldots\x\Am_{n-1}$.
  Precomposition is defined by recursion on paths. The base case is a
  path of length $0$:
  \[ \sem{\eps_{\pmobj{0}}} ~=~ \id_{\Ap_0\x\Am_0} \quad=\quad
  \mp{.4}{\includegraphics[scale=0.5]{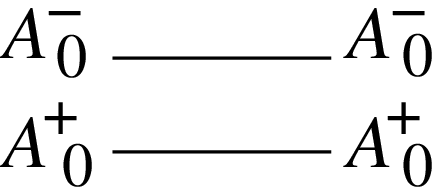}~.}
  \]
  And when $\vec p=p_1,\ldots,p_n$ as above is a path of length $n$,
  we define
  \[ \sem{\vec p,p_{n+1}} \quad=\quad
  \mp{.4}{\includegraphics[scale=0.5]{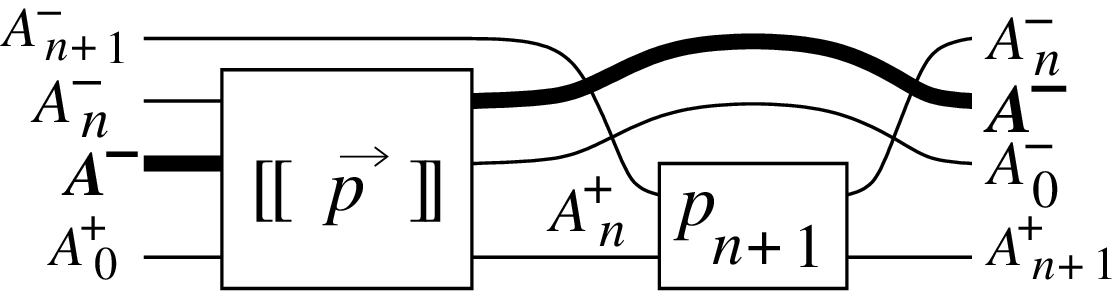}~.}
  \]
  Here, a thick line represents the object $\vecc\Am$, which really
  consists of $n$ parallel lines.
\end{definition}

\begin{definition}\label{def-intp-comp}
  For any path $\vec p=p_1,\ldots,p_n$, with $n\geq 0$ and
  $p_i:\pmobj{i-1}\to\pmobj{i}$, the partial composition $[\vec p\,]$ is
  defined as
  \[ [\vec p\,] ~:\kleeneeq~
  \Tr^{\vecc\Am}(\sem{\vec p}\circ(\Ap_0\x\sym_{\Am_n,\vecc\Am}))
  \quad\kleeneeq\quad
  \mp{.3}{\includegraphics[scale=0.5]{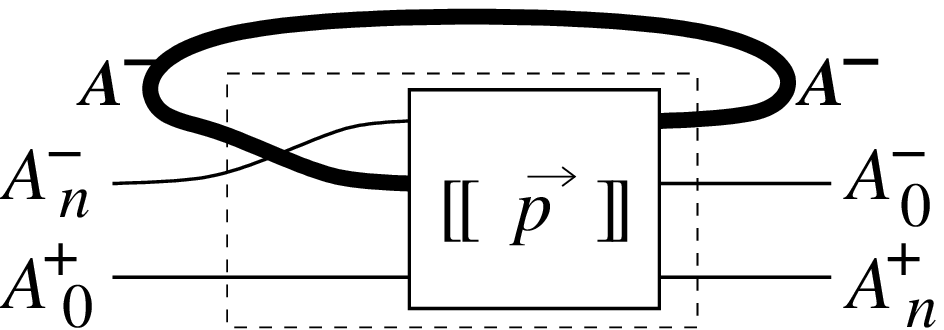}~.}
  \]
\end{definition}

The reader is invited to verify that in case $n=3$, this definition
indeed coincides with (\ref{eqn-partial-comp-len-3}).

\subsection{$\IntpC$ is a paracategory}

We start with a lemma that will be useful in the proof of the
paracategory properties for $\IntpC$. 

\begin{lemma}\label{lem-path-pq}
  For all paths $\vec p:\pmobjl{A}\to\pmobjl{B}$ and $\vec
  q:\pmobjl{B}\to\pmobjl{C}$,
  \[ \sem{\vec p,\vec q} \quad\kleeneeq\quad
  \mp{.4}{\includegraphics[scale=0.5]{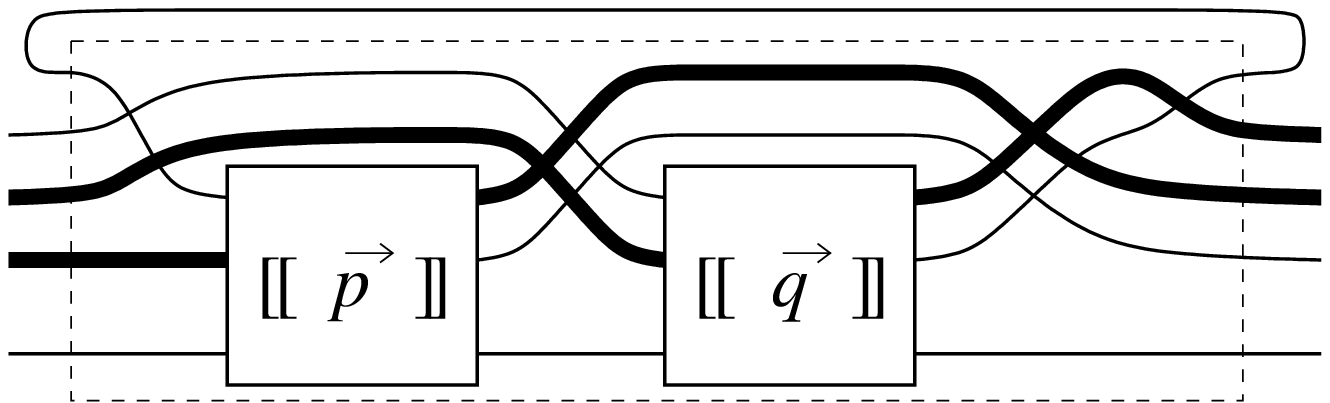}}~.
  \]
  In particular, the diagram is always defined.
\end{lemma}

\begin{proof}
  Since the left-hand side is always defined, it suffices to prove
  ``$\kleeneleq$''. We do this by induction on $\vec q$. For the base
  case, we have by yanking, strength, and naturality:
  \[  \sem{\vec p}
  ~\kleeneleq~
  \mp{.3}{\includegraphics[scale=0.5]{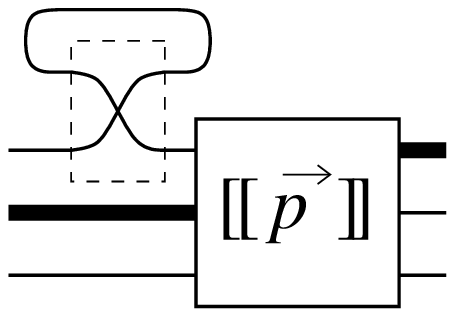}}
  ~\kleeneleq~
  \mp{.3}{\includegraphics[scale=0.5]{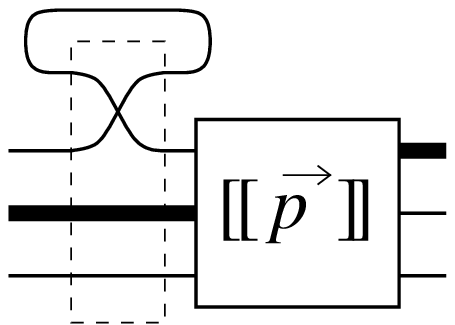}}
  ~\kleeneleq~
  \mp{.3}{\includegraphics[scale=0.5]{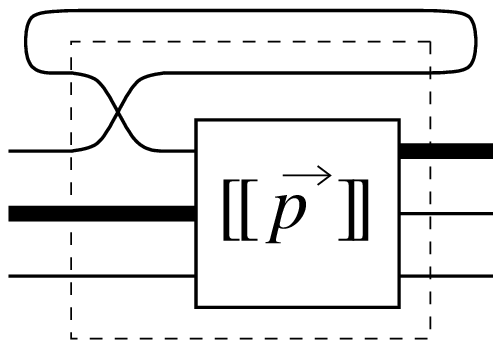}}
  \]\[
  ~\kleeneeq~
  \mp{.3}{\includegraphics[scale=0.5]{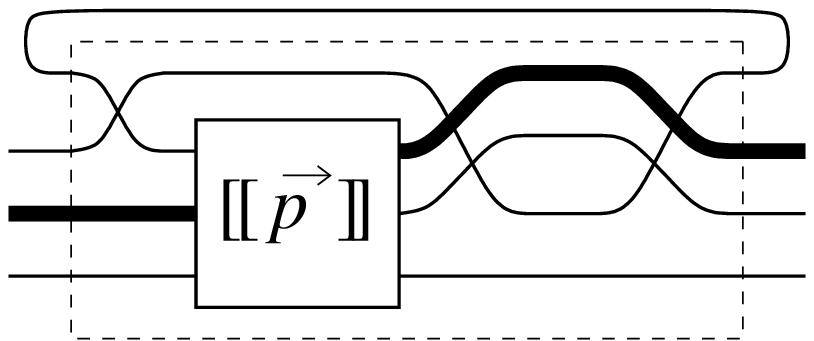}}
  ~\stackrel{\rm(def)}{\kleeneeq}~
  \mp{.3}{\includegraphics[scale=0.5]{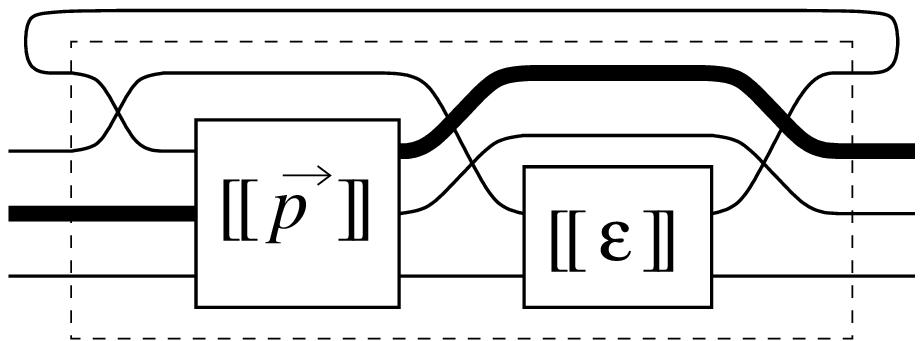}}.
  \]
  For the induction step, we have by superposing and naturality:
  \[
  \sem{\vec p,\vec q,q_{n+1}}
  ~\stackrel{\rm(def)}{\kleeneeq}~
  \mp{.3}{\includegraphics[scale=0.375]{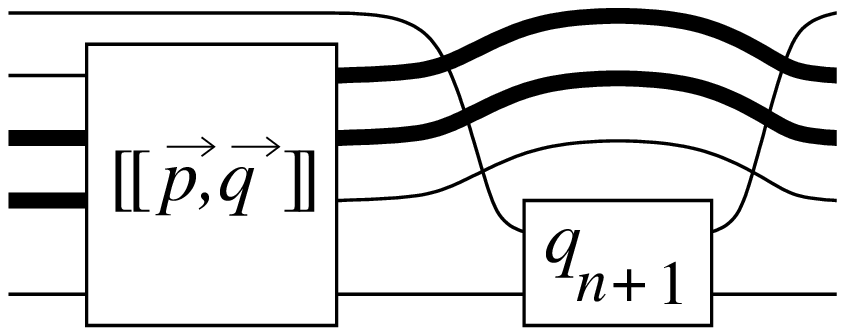}}
  ~\stackrel{\rm(ind.hyp.)}{\kleeneeq}~
  \mp{.3}{\includegraphics[scale=0.375]{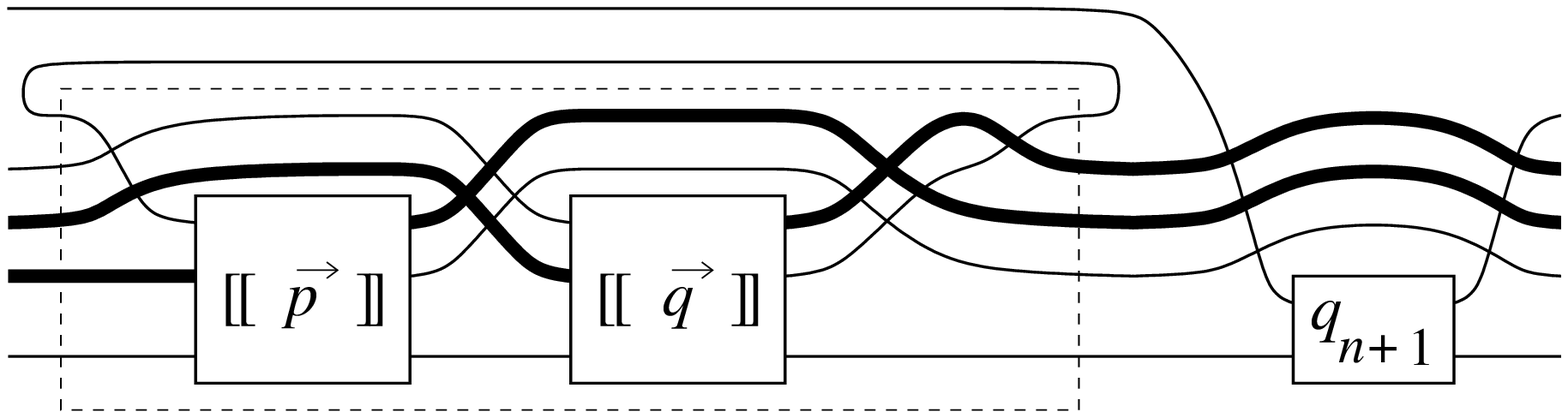}}
  \]\[
  ~\kleeneleq~
  \mp{.3}{\includegraphics[scale=0.375]{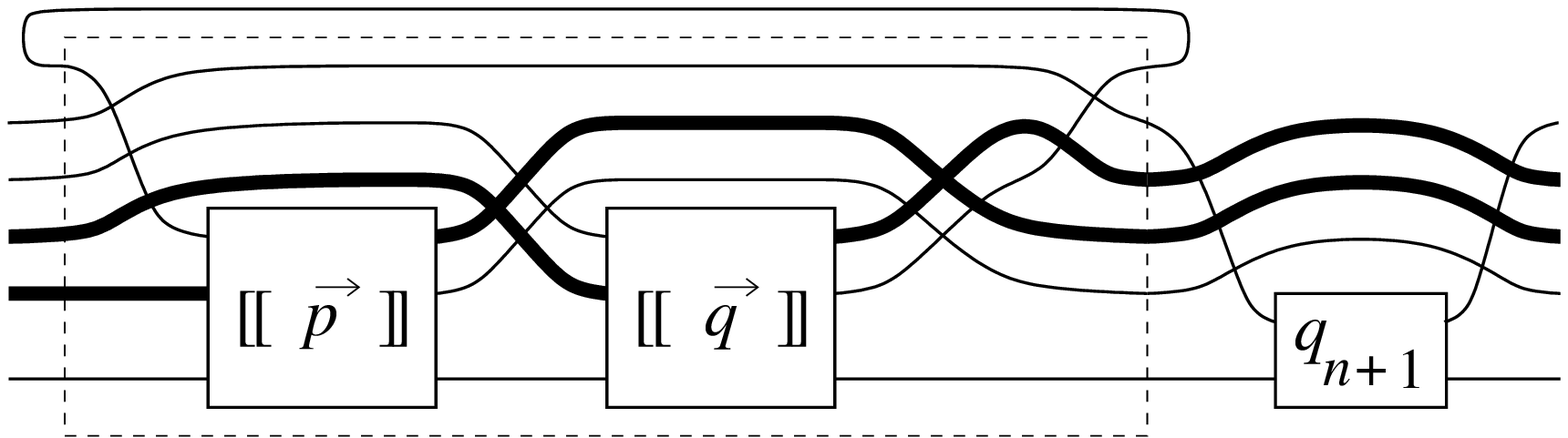}}
  ~\kleeneleq~
  \mp{.3}{\includegraphics[scale=0.375]{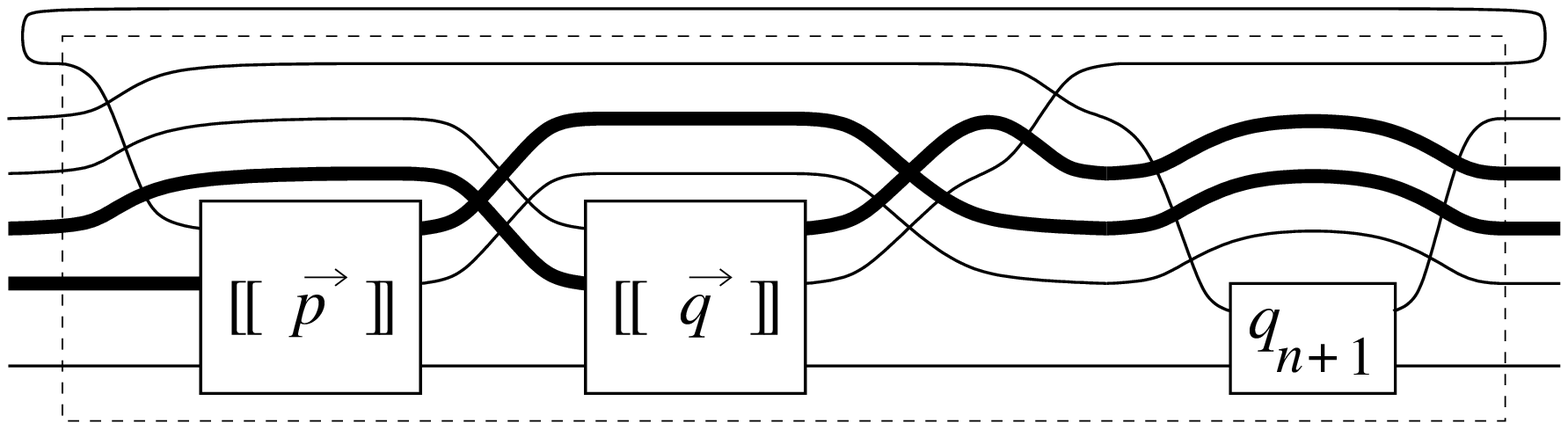}}
  \]\[
  ~\kleeneeq~
  \mp{.3}{\includegraphics[scale=0.375]{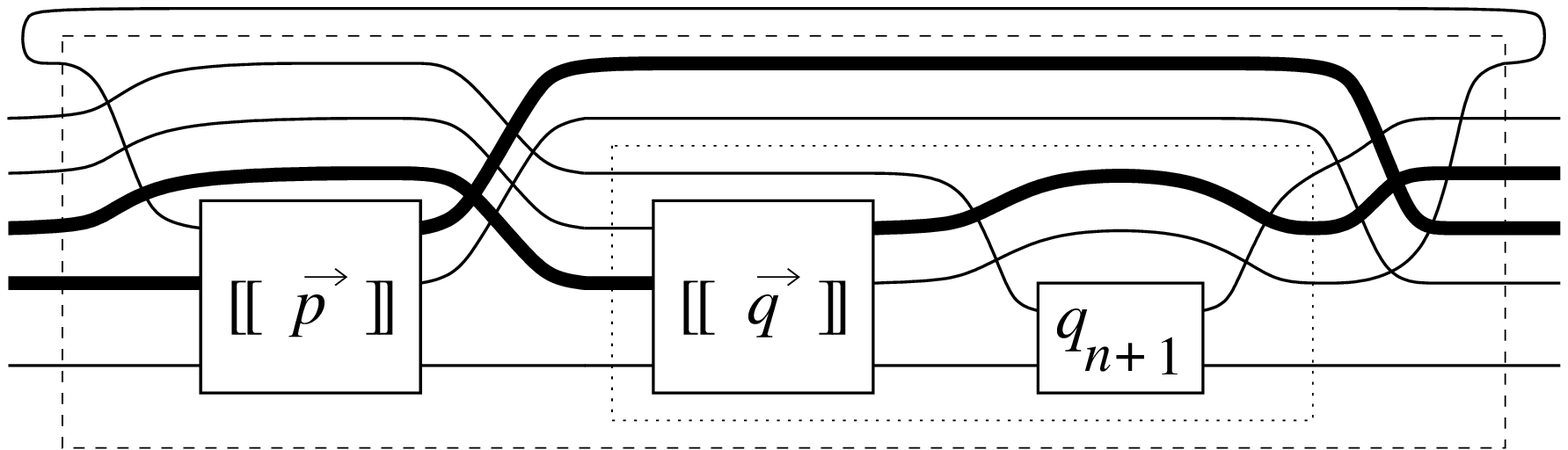}}
  ~\stackrel{\rm(def)}{\kleeneeq}~
  \mp{.3}{\includegraphics[scale=0.375]{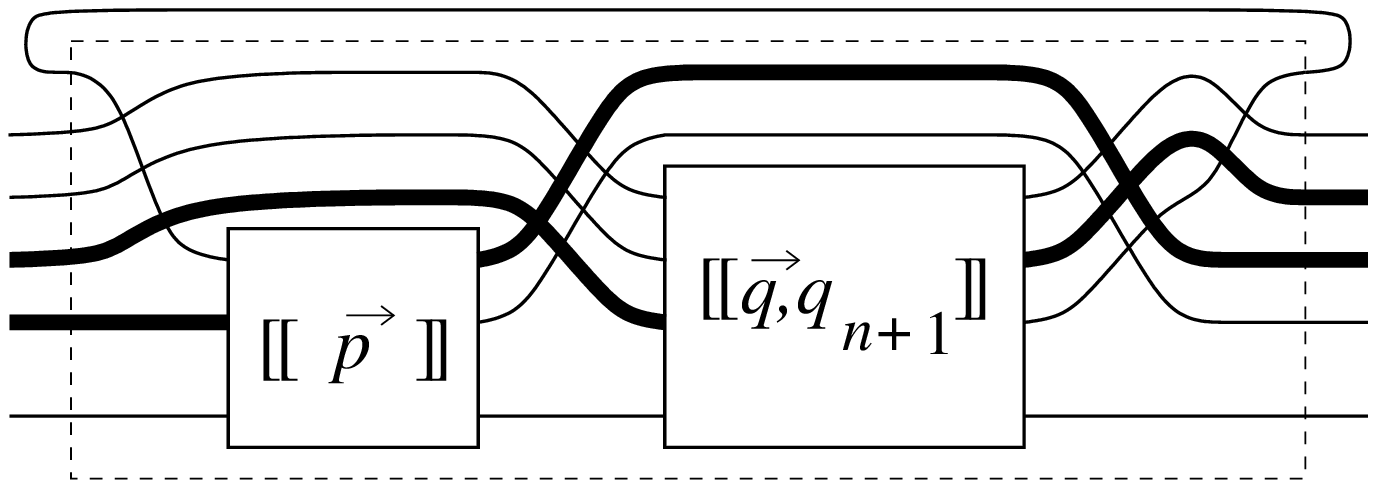}}.
  \]
\end{proof}

The above proof illustrates that the strength, superposing, and
naturality axioms all serve to ``enlarge'' the dashed boxes under
directed Kleene equality. To save space, in the following we often
combine these axioms, as well as the left-to-right direction
of vanishing II, into a single graphical step.

\begin{lemma}\label{lem-intpc-paracat}
  Let $\cC$ be a partially traced symmetric (strictly) monoidal
  category. With the partial composition $[-]$ defined in
  Definition~\ref{def-intp-comp}, $\IntpC$ is a paracategory.
\end{lemma}

\begin{proof}
  \begin{itemize}
  \item[(a)] By vanishing I, it follows immediately that
    $[\epsilon_{\pmobj{}}] = \sem{\epsilon_{\pmobj{}}} =
    \id_{\pmobj{}}$. In particular, $[\epsilon_{\pmobj{}}]\defined$.
  \item[(b)] For a path $f:\pmobj{0}\to\pmobj{1}$ of length 1, we have
    by yanking, strength, and naturality:
    \[
    \raisebox{-.8em}{\includegraphics[scale=0.5]{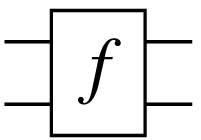}}
    ~\kleeneeq~
    \raisebox{-.8em}{\includegraphics[scale=0.5]{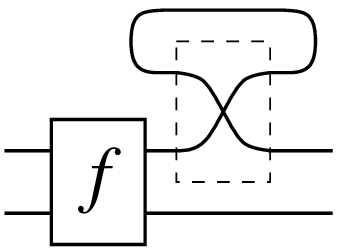}}
    ~\kleeneleq~
    \raisebox{-1em}{\includegraphics[scale=0.5]{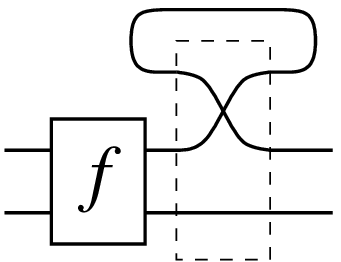}}
    ~\kleeneleq~
    \raisebox{-1em}{\includegraphics[scale=0.5]{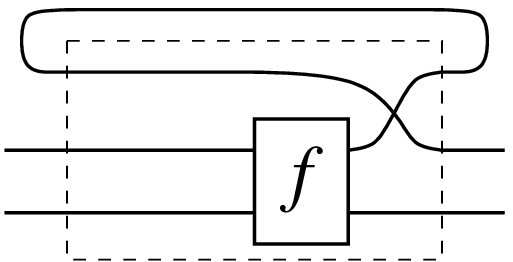}}
    ~\kleeneeq~
    \raisebox{-1em}{\includegraphics[scale=0.5]{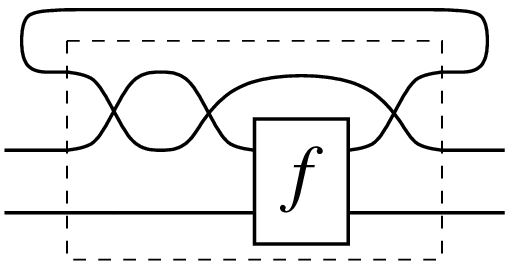}}
    ~\stackrel{\rm(def)}{\kleeneeq}~
    [f].
    \]
    In particular, the right-hand side is defined.
  \item[(c)] We must show that whenever $[\vec{q}\,]$ is defined, then
    $[\vec{p},[\vec{q}\,],\vec{r}\,]\kleeneeq[\vec{p},\vec{q},\vec{r}\,]$.
    First, by Lemma~\ref{lem-path-pq}, superposing, naturality, and
    vanishing II, we have
    \[ \sem{\vec{p},\vec{q},\vec{r}} \circ (\id\x\sym)
    ~\kleeneeq~
    \mp{0.4}{\includegraphics[scale=0.35]{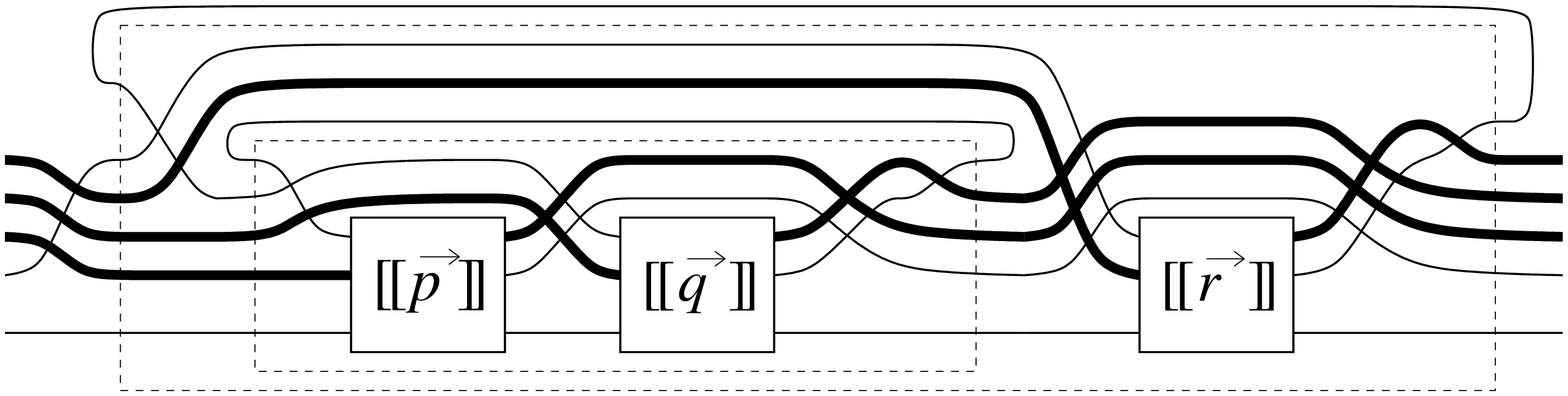}}
    \]\begin{equation}\label{eqn-intpc-paracat-c2}
      ~\kleeneleq~
      \mp{0.4}{\includegraphics[scale=0.35]{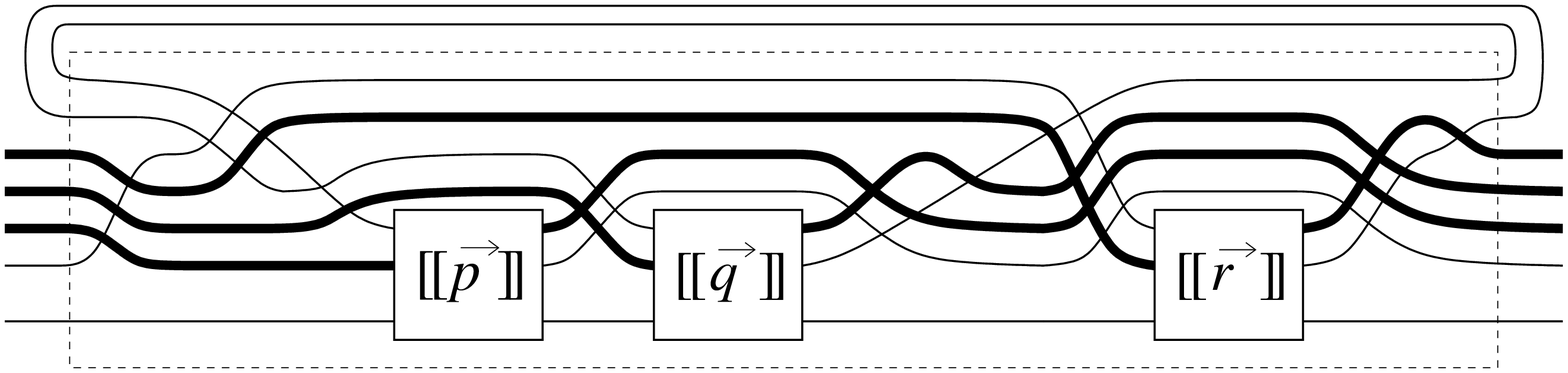}}~.
    \end{equation}
    Second, assume that $[\vec q\,]$ is defined. By definition of $[\vec
    q\,]$, Lemma~\ref{lem-path-pq}, superposing, naturality, and
    vanishing II, we have
    \[ \sem{\vec{p},[\vec{q}\,],\vec{r}} \circ (\id\x\sym)
    ~\kleeneeq~
    \mp{0.4}{\includegraphics[scale=0.35]{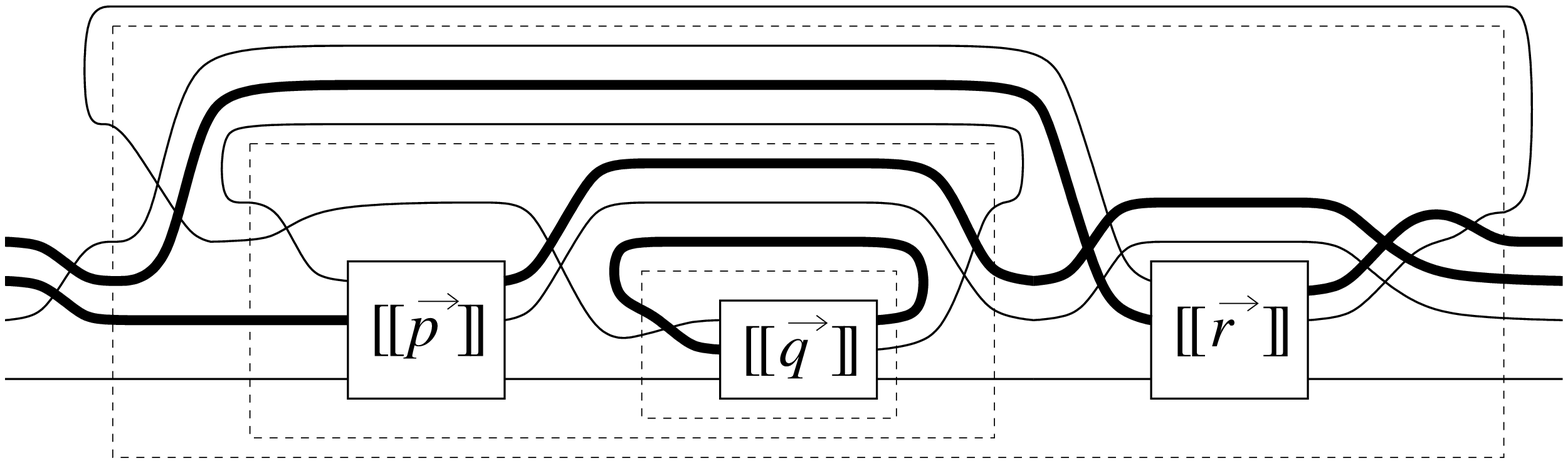}}
    \]\begin{equation}\label{eqn-intpc-paracat-c4}
      ~\kleeneleq~
      \mp{0.4}{\includegraphics[scale=0.35]{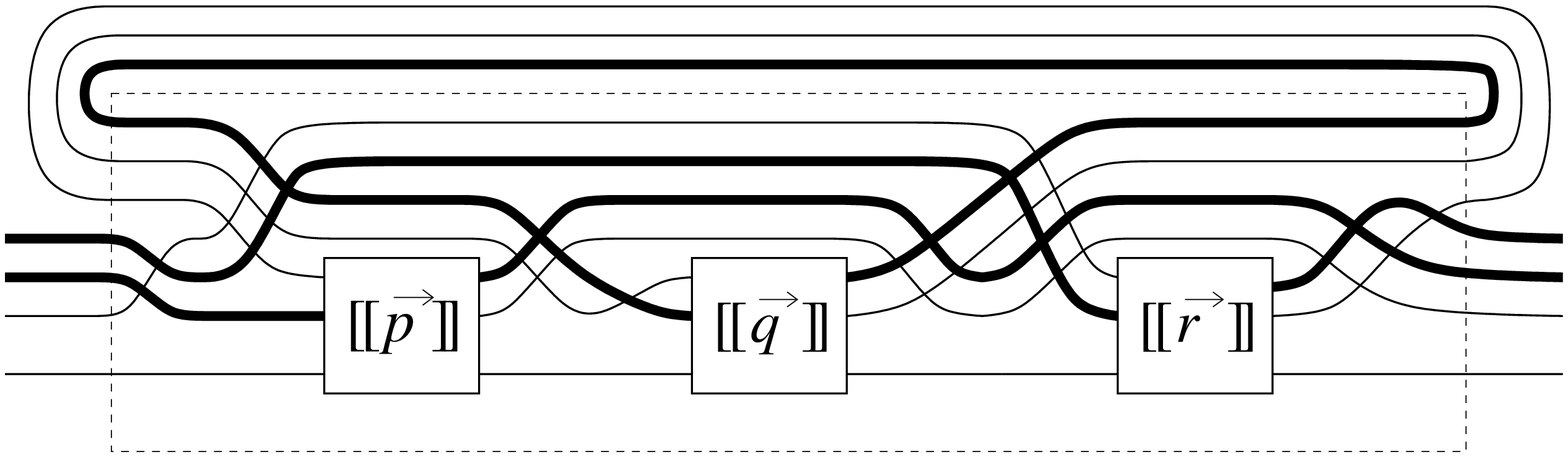}}~.
    \end{equation}
    Note that every morphism mentioned so far is defined. Recall that
    by definition, $[\vec{p},\vec{q},\vec{r}\,]$ and
    $[\vec{p},[\vec{q}\,],\vec{r}\,]$ are the trace of
    (\ref{eqn-intpc-paracat-c2}) and (\ref{eqn-intpc-paracat-c4}),
    respectively, where the trace is taken on the ``fat'' wires.  The
    fact that $[\vec{p},\vec{q},\vec{r}\,]\kleeneeq
    [\vec{p},[\vec{q}\,],\vec{r}\,]$ then follows immediately from
    vanishing II and dinaturality.
\end{itemize}
\end{proof}

\begin{lemma}\label{lem-comp-fg}
  For paths of length 2, we have
  \[ [f,g]
  ~\kleeneeq~
  \mp{.3}{\includegraphics[scale=0.5]{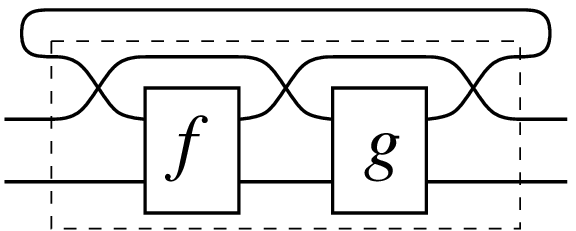}}.
  \]
\end{lemma}

\begin{proof}
  By yanking, strength, and naturality, we have
  \begin{equation}\label{eqn-comp-fg}
    \mp{.3}{\includegraphics[scale=0.5]{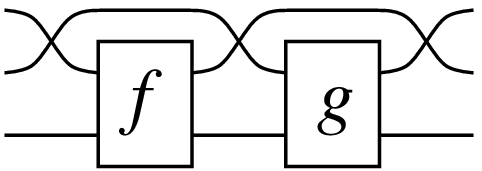}}
    ~\kleeneleq~
    \mp{.3}{\includegraphics[scale=0.5]{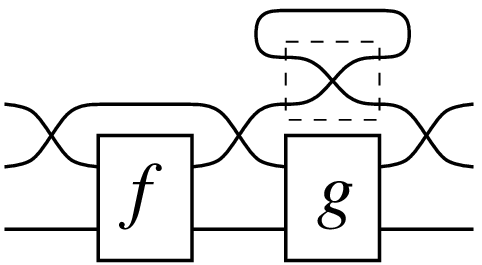}}
    ~\kleeneleq~
    \mp{.3}{\includegraphics[scale=0.5]{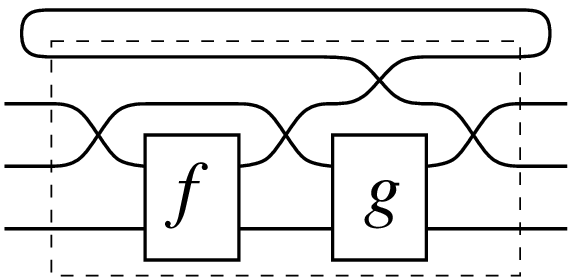}}.
  \end{equation}
  Since the left-hand side is defined, so is the right-hand side.
  This justifies the application of vanishing II in the following:
  \[
  \mp{.3}{\includegraphics[scale=0.5]{comp-fg8}}
  ~\stackrel{\rm(\ref{eqn-comp-fg})}{\kleeneeq}~
  \mp{.3}{\includegraphics[scale=0.5]{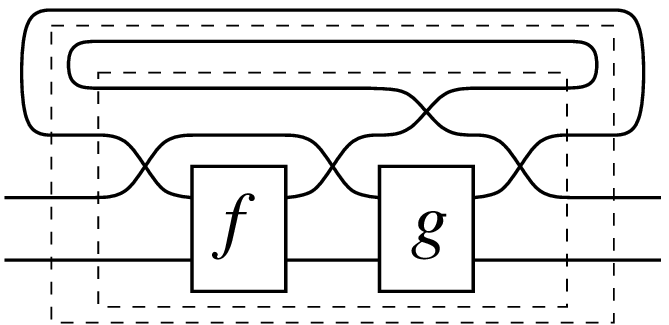}}
  ~\stackrel{\rm(v.II)}{\kleeneeq}~
  \mp{.3}{\includegraphics[scale=0.5]{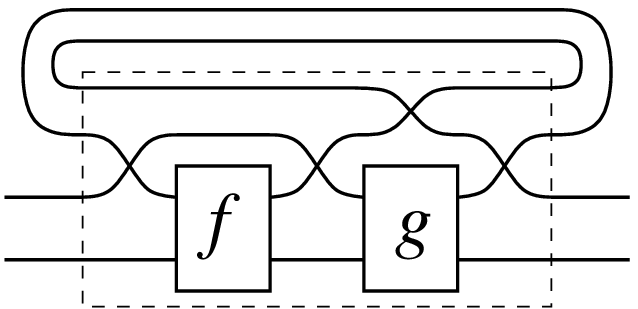}}
  \]\[
  ~\stackrel{\rm(dinat.)}{\kleeneeq}~
  \mp{.3}{\includegraphics[scale=0.5]{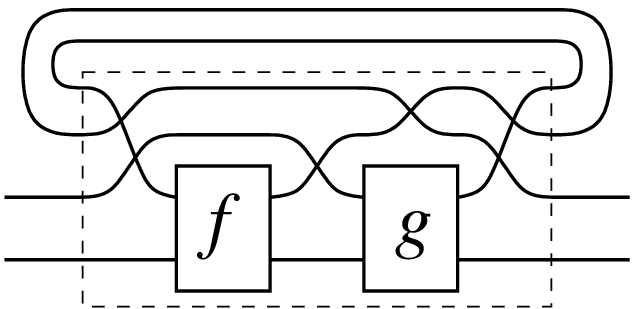}}
  ~\kleeneeq~
  \mp{.3}{\includegraphics[scale=0.5]{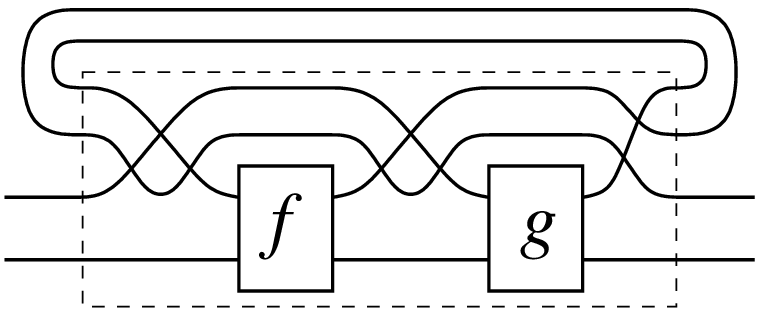}}
  ~\stackrel{\rm(def)}{\kleeneeq}~
  [f,g]
  .
  \]
\end{proof}

\subsection{$\IntpC$ is symmetric monoidal}

Next, we wish to show that the paracategory $\IntpC$ is strictly
monoidal.

\begin{definition}\label{def-intpc-tensor}
  The tensor on the paracategory $\IntpC$ is defined as follows:
  \begin{itemize}
  \item on objects: $(A^+,A^-)\otimes (B^+,B^-)=(A^+\otimes
    B^+,B^-\otimes A^-)$;
  \item on arrows: given $f^{\Intp}:(A^+,A^-)\rightarrow (C^+,C^-)$
    and $g^{\Intp}:(B^+,B^-)\rightarrow (D^+,D^-)$, then $(f\otimes
    g)^{\Intp}:(A^+,A^-)\otimes (B^+,B^-)\rightarrow
    (C^+,C^-)\otimes (D^+,D^-)$ is defined by
    \[ \includegraphics[scale=0.5]{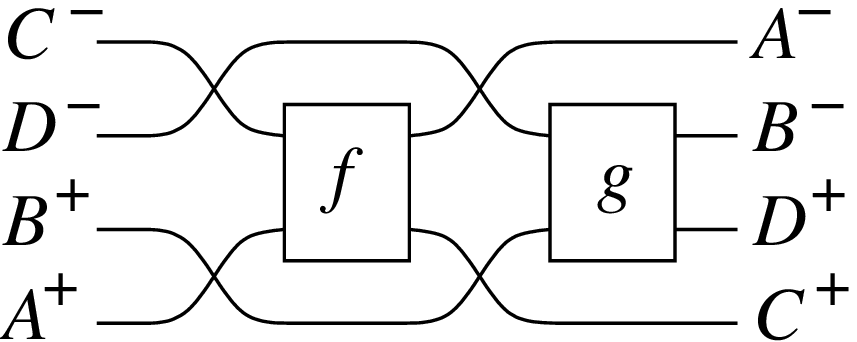}.
    \]
  \end{itemize}
  We also define the tensor unit to be $(I,I)$.
\end{definition}

\begin{lemma}
  The operation $\otimes$ is a functor of paracategories.
\end{lemma}

\begin{proof}
  We have to show the two conditions from Lemma~\ref{lem-ssmpc}.
  \begin{itemize}
  \item[(a)] We show $[f,f']\otimes [g,g']\kleeneleq [f\otimes g,
    f'\otimes g']$. By Lemma~\ref{lem-comp-fg}, strength, superposing,
    naturality, the left-to-right direction of vanishing, and the laws
    of symmetric monoidal categories, we have:
    \[ [f,f']\otimes [g,g']
    ~\kleeneeq~
    \mp{.3}{\includegraphics[scale=0.5]{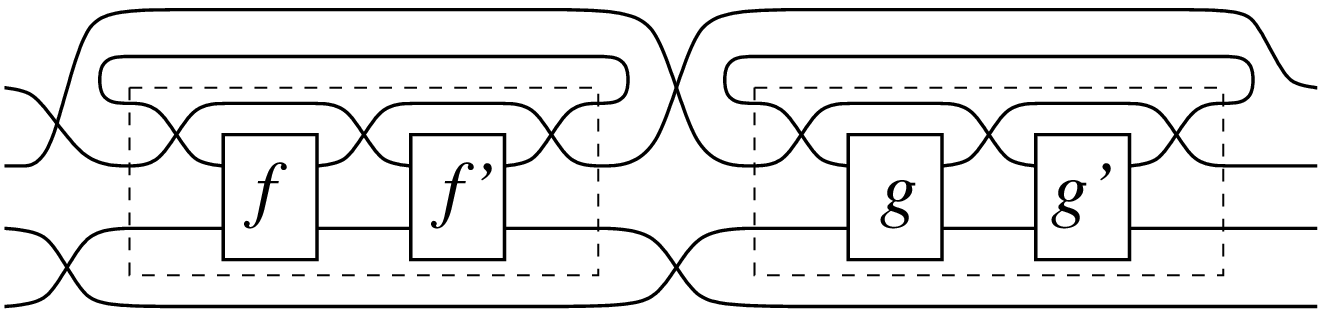}}
    \]\[
    ~\kleeneleq~
    \mp{.3}{\includegraphics[scale=0.5]{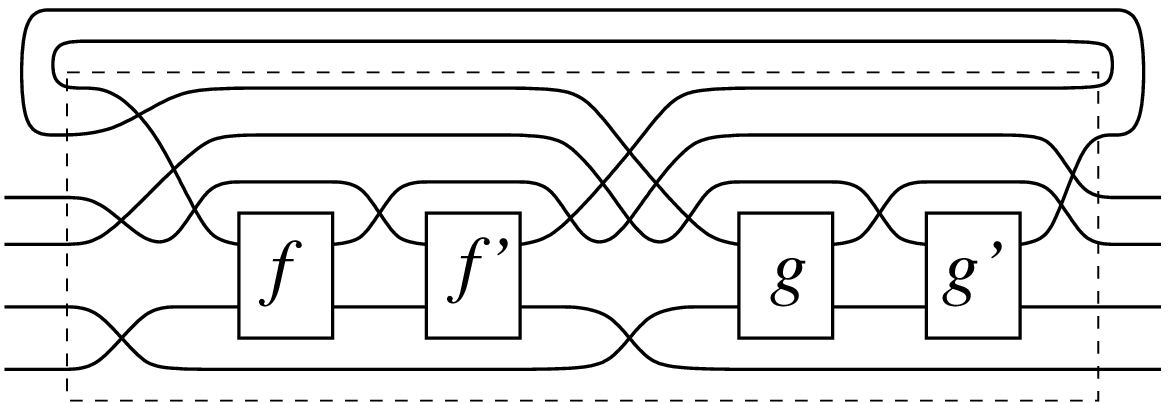}}
    ~\kleeneeq~
    \mp{.3}{\includegraphics[scale=0.5]{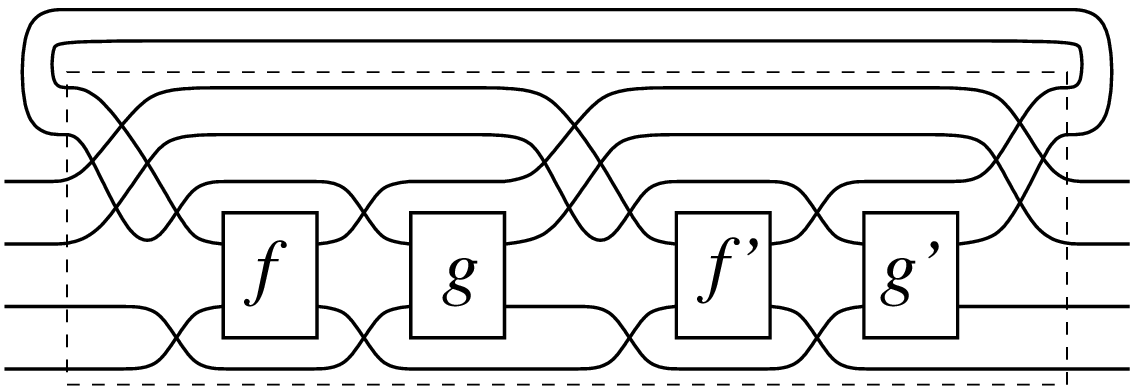}},
    \]
    and the final diagram is just the definition of $[f\otimes g,
    f'\otimes g']$.
  \item[(b)] We will only show $1\otimes [\,\vec{p}\,]\kleeneleq
    [\,1\otimes \vec{p}\,]$; the proof of the other property
    $[\,\vec{p}\,]\x 1\kleeneleq [\,\vec{p}\x 1\,]$ is similar. Since
    the proof by induction is long and not very interesting, we will
    only consider the representative case when
    $\vec{p}=p_1,p_2,p_3$. Using superposing, yanking, strength,
    naturality, vanishing II, and dinaturality, we have:
    \[ [\,\vec{p}\,]\x 1
    ~\stackrel{\rm(def)}{\kleeneeq}~
    \mp{.15}{\includegraphics[scale=0.4]{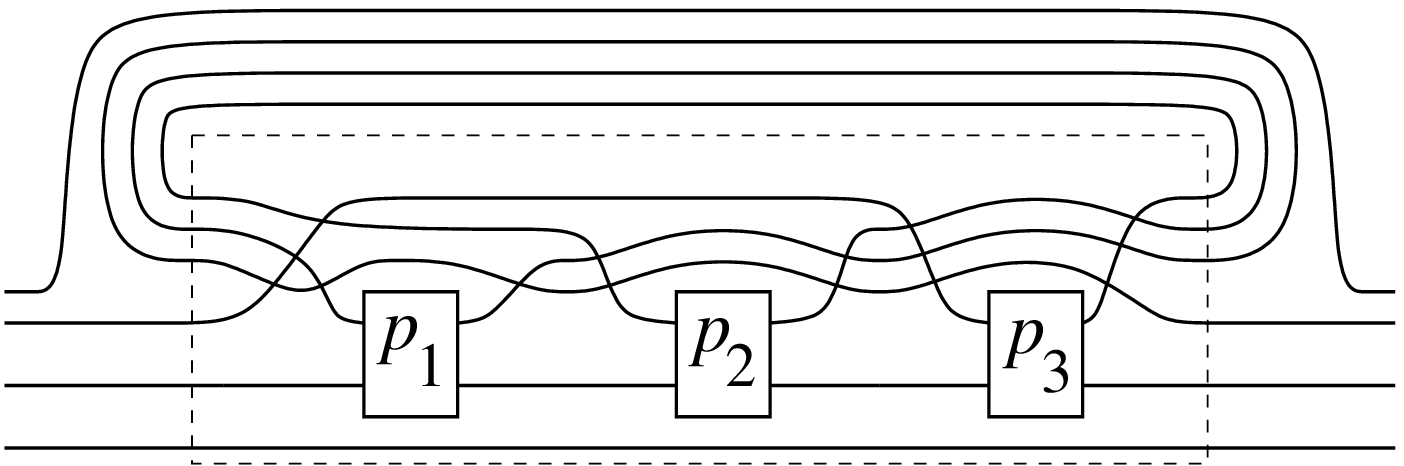}}
    ~\kleeneleq~
    \mp{.15}{\includegraphics[scale=0.4]{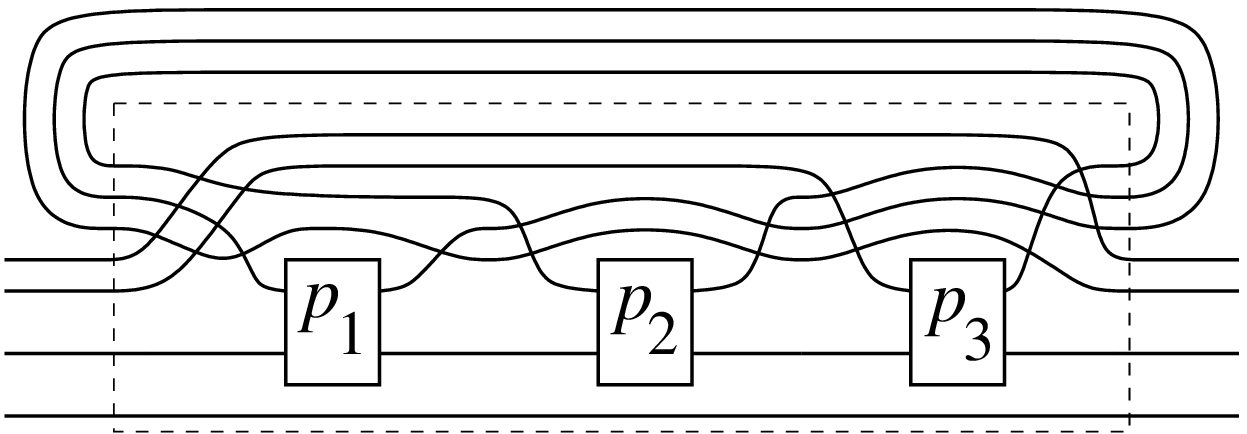}}
    \]\[
    ~\kleeneeq~
    \mp{.15}{\includegraphics[scale=0.4]{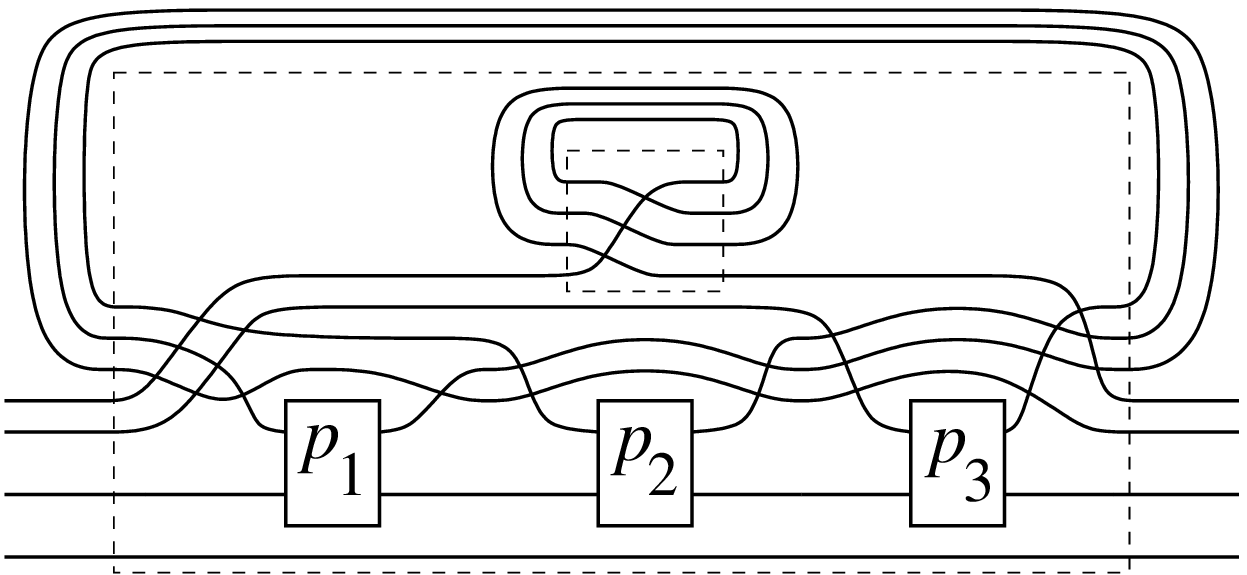}}
    ~\kleeneleq~
    \mp{.15}{\includegraphics[scale=0.4]{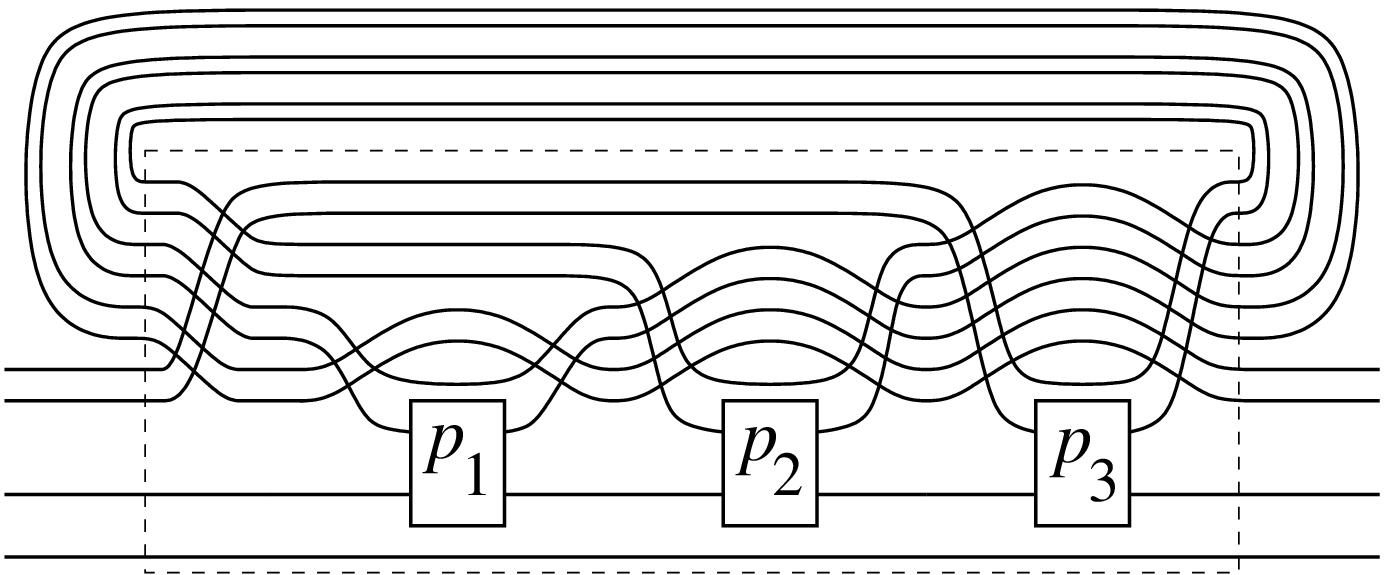}}
    ~\stackrel{\rm(def)}{\kleeneeq}~
    [\,\vec{p}\otimes 1\,].
    \]
  \end{itemize}
\end{proof}

\begin{lemma}
  With the tensor product from Definition~\ref{def-intpc-tensor},
  $\IntpC$ is a strict monoidal paracategory in the sense of
  Definition~\ref{def-ssmpc}(b).
\end{lemma}

\begin{proof}
  The conditions $(A\otimes B)\otimes C=A\otimes (B\otimes C)$,
  $A\otimes I=A=I\otimes A$, and $f\otimes 1_I=f=1_I\otimes f$ follow
  immediately from the strictness of $\cC$. The condition $(f\otimes
  g)\otimes h=f\otimes (g\otimes h)$ holds because both sides are
  equal to the diagram
  \[    \includegraphics[scale=0.5]{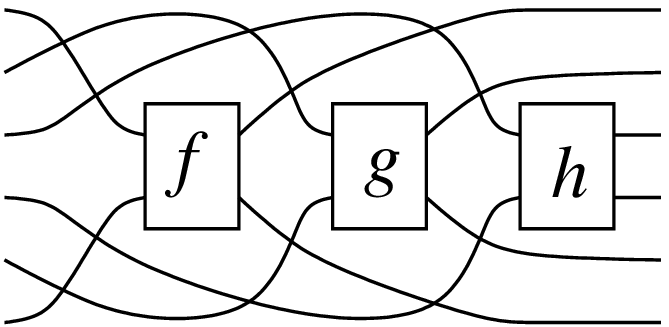}.
  \]
\end{proof}


Next, we will equip the category $\IntpC$ with a symmetry.

\begin{definition}
  \rm The symmetry $\sym:(\Ap,\Am)\otimes (\Bp,\Bm)\to
  (\Bp,\Bm)\otimes (\Ap,\Am)$ in $\IntpC$ is given by
  $\sym_{\Ap,\Bp}\otimes
  \sym_{\Am,\Bm}:(\Ap\x\Bp)\x(\Bm\x\Am)\to(\Bp\x\Ap)\x(\Am\x\Bm)$.
\end{definition}

\begin{lemma}
  With this structure, $\IntpC$ is a strict symmetric monoidal paracategory.
\end{lemma}

\begin{proof}
  We must show that $\sym$ satisfies the conditions of
  Definition~\ref{def-ssmpc}(c).  To prove totality, consider any
  $f:X\x B\x A\to Y$, where $A=\pmobjl{A}$, $B=\pmobjl{B}$,
  $X=\pmobjl{X}$, and $Y=\pmobjl{Y}$. We must prove that
  $[1_X\x\sym_{A,B},f]$ is defined. But using yanking, strength,
  naturality, and Lemma~\ref{lem-comp-fg}, we have
  \begin{equation}\label{eqn-sym-tot4}
    \mp{.3}{\includegraphics[scale=0.5]{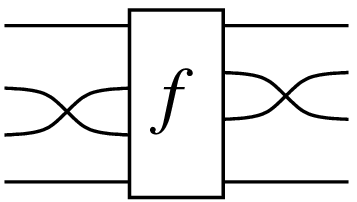}}
    ~\kleeneeq~
    \mp{.15}{\includegraphics[scale=0.5]{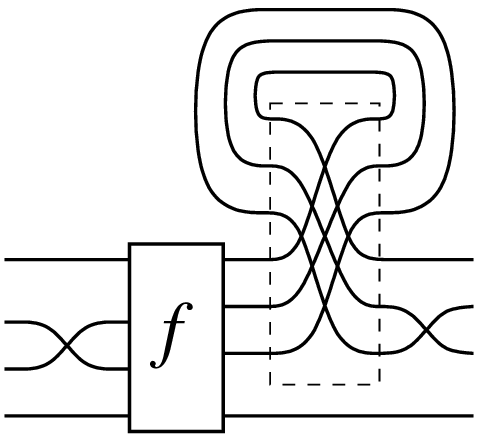}}
    ~\kleeneleq~
    \mp{.18}{\includegraphics[scale=0.5]{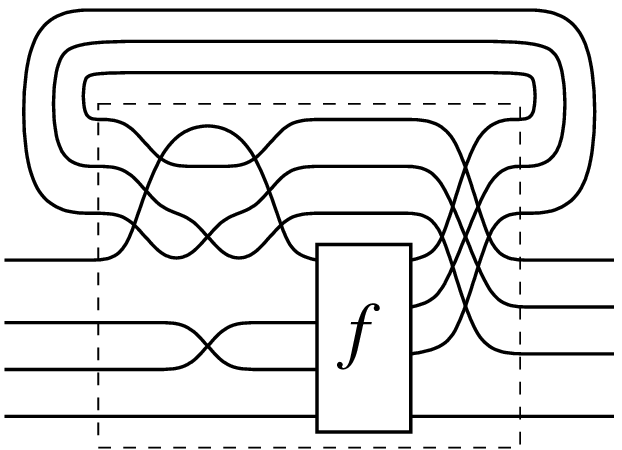}}
    ~\kleeneeq~
    [1_X\x\sym_{A,B},f].
  \end{equation}
  Since the left-hand side is defined, so is the right-hand side. One
  similarly proves that $[g,1_X\x \sigma_{A,B}]\defined$. By setting
  $X=1$ in (\ref{eqn-sym-tot4}) and the corresponding property for
  $g$, we get the identities
  \begin{equation}\label{eqn-sym-tot5}
    [\sym_{A,B},f]
    ~=~
    \mp{.3}{\includegraphics[scale=0.5]{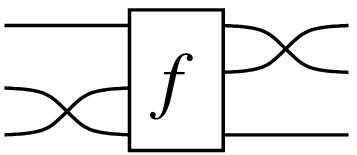}}
    \quad\mbox{and}\quad
    [g,\sym_{A,B}]
    ~=~
    \mp{.3}{\includegraphics[scale=0.5]{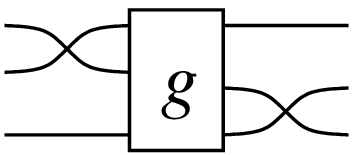}}.
  \end{equation}
  The remaining laws follow from (\ref{eqn-sym-tot5}). We have:
  \begin{equation}\label{eqn-intpc-sym-nat}
    [f\otimes g,\sigma]
    ~\stackrel{\rm(\ref{eqn-sym-tot5})}{=}~
    \mp{.4}{\includegraphics[scale=0.5]{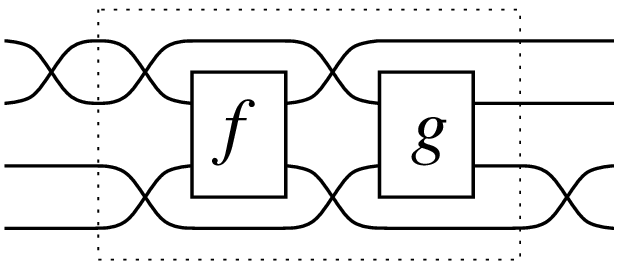}}
    ~=~
    \mp{.4}{\includegraphics[scale=0.5]{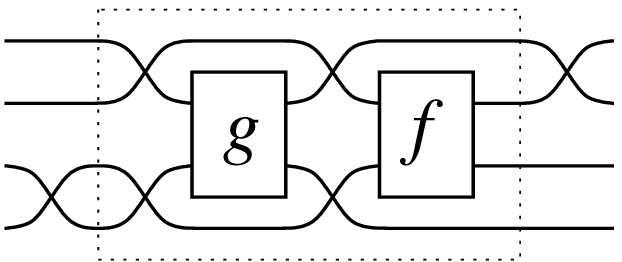}}
    ~\stackrel{\rm(\ref{eqn-sym-tot5})}{=}~
    [\sigma,g\x f].
  \end{equation}
  \begin{equation}\label{eqn-intpc-sym-sym}
    [\sigma,\sigma]
    ~\stackrel{\rm(\ref{eqn-sym-tot5})}{=}~
    \mp{.4}{\includegraphics[scale=0.5]{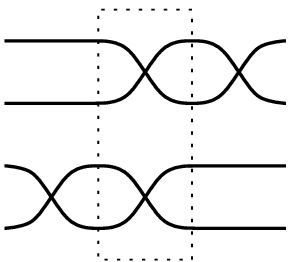}}
    ~=~ \id
  \end{equation}
  \begin{equation}\label{eqn-intpc-sym-hex}
    [\sym_{A,B\x C},1_B\x\sym_{C,A}]
    ~\stackrel{\rm(\ref{eqn-sym-tot5})}{=}~
    \mp{.4}{\includegraphics[scale=0.5]{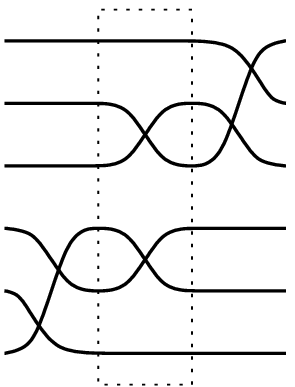}}
    ~\kleeneeq~
    \mp{.4}{\includegraphics[scale=0.5]{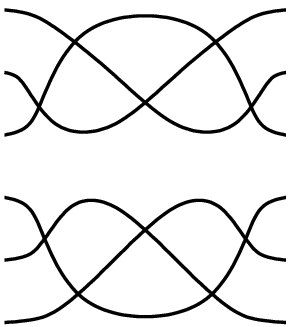}}
    ~\kleeneeq~
    \sym_{A,B}\x 1_C.
  \end{equation}
  Naturality is (\ref{eqn-intpc-sym-nat}), symmetry is
  (\ref{eqn-intpc-sym-sym}), and the hexagon axiom is equivalent to
  (\ref{eqn-intpc-sym-hex}) by Remark~\ref{rem-inverse}.
\end{proof}

\subsection{$\IntpC$ is compact closed}

\begin{definition}
  On $\Intp(C)$, we define the dual of an object to be $(A,B)^* =
  (B,A)$. Using strictness, we define the unit and counit morphisms
  $\eta_{(A,B)}:I\to (A,B)^*\x(A,B)$ and
  $\eps_{(A,B)}:(A,B)\x(A,B)^*\to I$ to be the morphisms $\eta_{(A,B)}=\id:B\x
  A\to B\x A$ and as $\eps_{(A,B)}=\id:A\x B\to A\x B$ in $\cC$.
\end{definition}

\begin{lemma}
  With this structure, $\IntpC$ is a compact closed paracategory.
\end{lemma}

\begin{proof}
  Let $f:A\x C^*\to B$. We must show that $[1_A\x\eta_C, f\x 1_C]$ is
  defined. We have:
  \[
  f
  ~\kleeneeq~
  \mp{.13}{\includegraphics[scale=0.5]{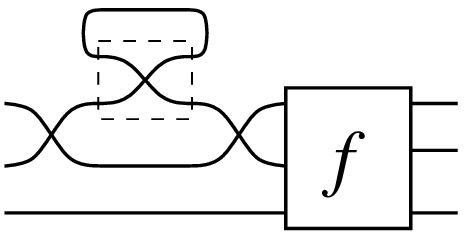}}
  ~\kleeneleq~
  \mp{.27}{\includegraphics[scale=0.5]{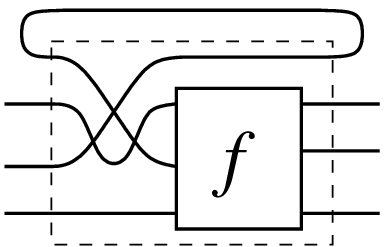}}
  ~\kleeneeq~
  \mp{.27}{\includegraphics[scale=0.5]{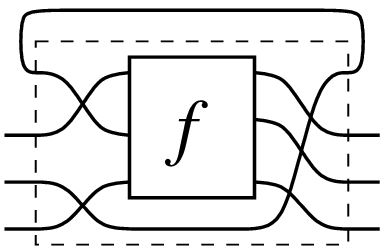}}
  ~\kleeneeq~
  \mp{.15}{\includegraphics[scale=0.5]{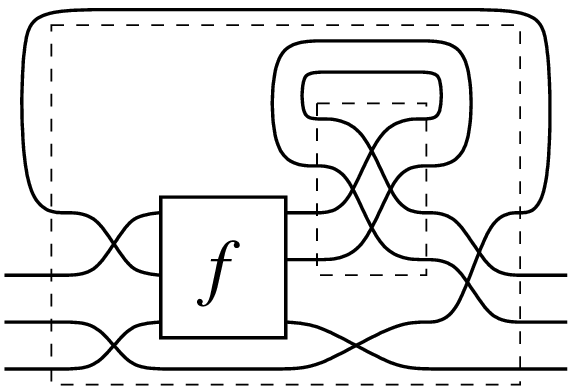}}
  \]\[
  ~\kleeneleq~
  \mp{.16}{\includegraphics[scale=0.5]{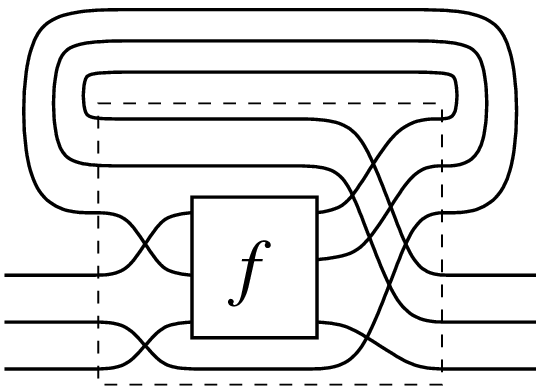}}
  ~\kleeneeq~
  \mp{.15}{\includegraphics[scale=0.5]{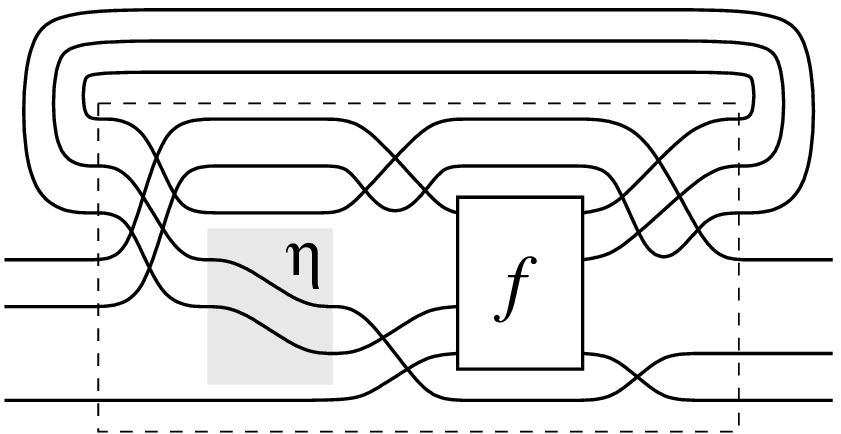}}
  ~\kleeneeq~
  [1_A\x\eta_C, f\x 1_C],
  \]
  and since the left-hand side is defined, so is the right-hand
  side. The proofs for the definedness of $[g\x 1_{C^*},
  1_B\x\eps_C]$, $[\eta_A\x 1_B,1_{A^*}\x h]$, and $[1_A\x k, \eps_A\x
  1_C]$ are similar. We prove that $[1_A\x \eta_A,\eps_A\x 1_A ]=1_A$
  by setting $C=A$, $B=I$, and $f=\eps_A$ in the above, and recalling
  that $\eps_A=\id_{\Ap\x\Am}=\id_A$ as morphisms of $\cC$. The proof
  of $[\eta_A\x 1_{A^*}, 1_{A^*}\x \eps_A]=1_{A^*}$ is analogous.
\end{proof}

\subsection{An embedding of $\cC$ in $\IntpC$}

\begin{definition}
  In a similar way as done in~\cite{JSV96}, we define a full and
  faithful functor of paracategories $N:\cC\rightarrow \IntpC$. It is
  given on objects by $N(A)=(A,I)$ and (using strictness of the
  category $\cC$) on morphisms by $N(f)=f$.
\end{definition}

\begin{theorem}\label{thm-n-faithful}
  $N$ is a full and faithful functor of strict symmetric monoidal
  paracategories. In particular, every partially traced (strictly
  monoidal) category can be fully and faithfully embedded in a compact
  closed paracategory.
\end{theorem}

\begin{proof}
  To prove functoriality, note that we are considering the category
  $\cC$ as a paracategory with composition $[f_1,\ldots,f_n] =
  f_n\circ\ldots\circ f_1$. It follows immediately from the definition
  of composition on $\IntpC$, strictness, and vanishing I that
  $[N(f_1),\ldots, N(f_n)] \kleeneeq [f_1,\ldots, f_n]
  = \sem{f_1,\ldots, f_n} = f_n\circ\ldots\circ f_1 =
  N(f_n\circ\ldots\circ f_1)$, so $N$ is a functor.  The fact that $N$
  is full and faithful is also obvious, as is the fact that it
  preserves tensor and symmetry.
\end{proof}

\begin{theorem}\label{thm-n-preserves-trace}
  The functor $N:\cC\rightarrow \IntpC$ preserves and reflects the
  trace, i.e., for all morphisms $f:A\otimes U\rightarrow B\otimes U$
  and $g:A\to B$ in $\cC$, we have $\Tr^U(f)=g$ iff ${\Tr^{NU}}N(f)=N(g)$.
\end{theorem}

\begin{proof}
  Recall that the trace on $\IntpC$ is defined as in
  Remark~\ref{rem-trace-cc-paracat}. Because $N$ is full and faithful,
  the claimed property is equivalent to $N(\Tr^U(f)) \kleeneeq
  {\Tr^{NU}}N(f)$. Using similar methods as in previous proofs, we
  have:
  \[ N(\Tr^U(f))
    ~\stackrel{\rm(def)}{\kleeneeq}~
    \Tr^U(f)
    ~\kleeneeq~
    \mp{.13}{\includegraphics[scale=0.5]{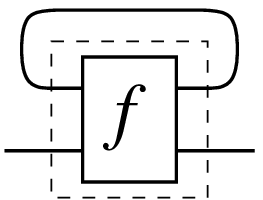}}
    ~\kleeneeq~
    \mp{.13}{\includegraphics[scale=0.5]{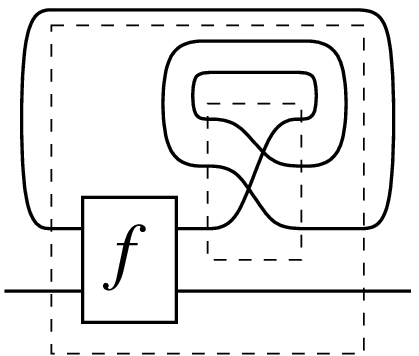}}
    ~\kleeneeq~
    \mp{.13}{\includegraphics[scale=0.5]{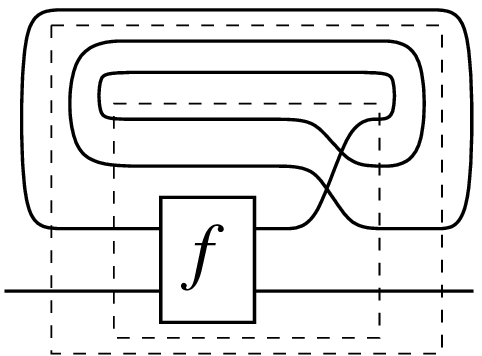}}
    \]\[
    ~\kleeneeq~
    \mp{.13}{\includegraphics[scale=0.5]{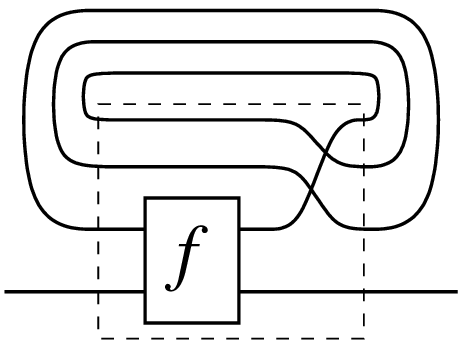}}
    ~\kleeneeq~
    \mp{.13}{\includegraphics[scale=0.5]{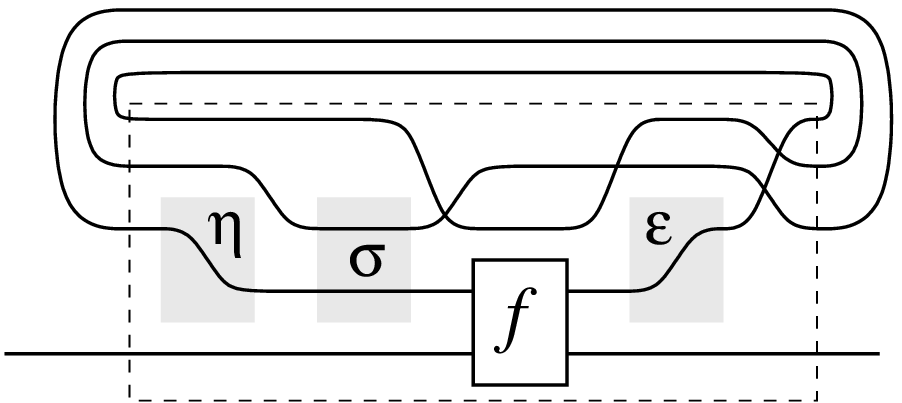}}
    ~\stackrel{\rm(def)}{\kleeneeq}~
    {\Tr^{NU}}N(f).
  \]
\end{proof}

\subsection{The universal property of $\IntpC$}

The category $\IntpC$ is in fact the free compact closed paracategory
over a given partially traced category. To be able to state this
theorem, we first need to define the notation of a (non-strict)
functor of compact closed paracategories.

\begin{definition}
  An isomorphism $m:A\to B$ in a symmetric monoidal 
  paracategory is said to be {\em total} if $[\id_C\x m, f]$, $[g,
  \id_C\x m]$, $[\id_C\x m^{-1}, h]$, and $[k, \id_C\x m^{-1}]$ are
  defined, for all $f:C\x B\to D$, $g:D\to C\x A$, $h:C\x A\to D$, and
  $k:D\to C\x B$.
\end{definition}

\begin{definition}
  Let $\cD$ and $\cD'$ be compact closed paracategories. A
  (non-strict) functor of compact closed paracategories $K:\cD\to\cD'$
  is a functor of paracategories that is equipped with total natural
  isomorphisms $m_{A,B}:K(A)\otimes'K(B)\rightarrow K(A\otimes B)$,
  $m_I:I'\rightarrow K(I)$, and $m_{*}:(KA)^* \to K(A^*)$,
  respecting all the structure.
\end{definition}

\begin{remark}
  In the presence of $m_{A,B}$ and $m_I$, a unique coherent
  isomorphism $m_{*}:(KA)^* \to K(A^*)$ automatically exists, but its
  totality is an additional property that must be required.
\end{remark}

\begin{theorem}\label{thm-univ-intpc}
  Let $\cC$ be a partially traced symmetric (strictly) monoidal
  category, $\cD$ a compact closed paracategory, and $G:\cC\to\cD$ a
  trace-preserving functor of symmetric monoidal paracategories. Then
  there exists an essentially unique (non-strict) functor of compact
  closed paracategories $K:\IntpC\to\cD$ such that
  \[ \xymatrix@C+5ex{
    \cC \ar[r]^<>(.5){N}\ar[dr]_{G} & \IntpC \ar[d]^{K} \\
    & \cD.
    }
  \]
\end{theorem}

\begin{proof}
  Without loss of generality we write as if $\cD$ were also strictly
  monoidal. Let us also write $G(A)=\bar A$.  The construction of the
  functor $K:\IntpC\rightarrow\cD$ is similar to that of Joyal,
  Street, and Verity in {\cite{JSV96}}. On objects, it is defined as
  $K(A,B) = \bar A\x \bar B^*$. A morphism $f:(A,B)\to(C,D)$ is given
  by $f: A\x D\to C\x B$ in $\cC$, and we have $G(f):\bar A\x \bar
  D\to \bar C\x \bar B$. Then $K(f):\bar A\x\bar B^*\to\bar C\x\bar
  D^*$ is defined as
  \[ K(f) := [\id_{\bar A}\x\eta_{\bar D}\x\id_{\bar B^*},~
  \id_{\bar A}\x\sym_{\bar D^*,\bar D}\x\id_{\bar B^*},~
  G(f)\x\sym_{\bar D^*,\bar B^*},~
  \id_{\bar C}\x\eps_{\bar B}\x\id_{\bar D^*}].
  \]
  It follows from the axioms of compact closed paracategories that
  $K(f)\defined$. The remaining properties are tedious but
  routine to verify.
\end{proof}

\begin{remark}
  Even when $\cC$, $\cD$, and $G$ are strict, one cannot in general
  expect $K$ to be strict. This is because the objects of the category
  $\IntpC$ satisfy special equations, such as $A\x B^* = B^*\x A$ for
  all $A,B$ in the image of $N$. Since one cannot expect $\cD$ to
  satisfy such equations, $K$ cannot in general be strictly monoidal.
\end{remark}

\section[Representation theorem]
{Representation theorem for partially traced categories}
\label{sec-representation}

By combining the results of the previous sections, we arrive at the
main theorem of this paper.

\begin{theorem}
  Every partially traced category can be faithfully embedded in a
  totally traced category. Moreover, the embedding is trace preserving
  and reflecting.
\end{theorem}

\begin{proof} Let $\cC$ be a partially traced category. We may without
  loss of generality assume that $\cC$ is strictly monoidal. By
  Theorems~\ref{thm-n-faithful} and~\ref{thm-n-preserves-trace}, there
  is a full and faithful, trace preserving and reflecting embedding
  $N:\cC\to\IntpC$ of $\cC$ in a compact closed paracategory. By
  Theorem~\ref{thm-faithful-embed-comp-closed-para}, there is a
  faithful embedding $F:\IntpC\to\path(\IntpC)/{\sim}$ of $\IntpC$ in
  a compact closed category. Since $\path(\IntpC)/{\sim}$ is compact
  closed, it is totally traced, and by Theorem~\ref{thm-cc-trace}, $F$
  is trace preserving and reflecting.
\end{proof}

\begin{corollary}
  Every partially traced category arises from a totally traced
  category by the construction of Proposition~\ref{prop-subcategory}.
\end{corollary}

\begin{corollary}
  Any equational law of totally traced categories also holds in all
  partially traced categories, provided that the left-hand side and
  right-hand side are both defined. In particular, reasoning in the
  graphical language of traced monoidal categories is sound for
  proving the equality of two morphisms in partially traced
  categories, provided both morphisms are defined. The morphisms used
  in intermediate steps do not need to be defined.
\end{corollary}

\begin{proof}
  Via the faithful embedding in a totally traced category, the
  reasoning really takes place in that category.
\end{proof}

Moreover, the category $\path(\IntpC)/{\sim}$ satisfies the following
universal property.

\begin{theorem}
  Let $\cC$ be a partially traced category and $\cD$ a compact closed
  category. If $G:\cC\rightarrow\cD$ is a traced symmetric monoidal
  functor, then there exists an essentially unique strong symmetric
  monoidal functor $L:\path(\IntpC)/{\sim}\rightarrow\cD$ such that
  \[ \xymatrix@=25pt{
    \cC\ar[r]^<>(.5){N}\ar[rrd]_{G}
    &\IntpC\ar[r]^<>(.5){F}
    & \path(\IntpC)/{\sim} \ar[d]^{L}\\
    && \cD.
  }
  \]
\end{theorem}

\begin{proof}
  By combining Theorems~\ref{thm-univ-intpc} and {\ref{thm-freeness}}.
  \[\xymatrix@=25pt{
    \cC\ar[rr]^{N}\ar[rrrrd]_{G}&&\IntpC\ar[rr]^{F}\ar[drr]^{K}& &\path(\IntpC)/{\sim} \ar[d]^{L}\\
    &&&& \cD.
  }\]
\end{proof}

\section{Discussion and future work}

We established that the partially traced categories, in the sense of
Haghverdi and Scott, are precisely the monoidal subcategories of
totally traced categories.  This was proved by a partial version of
Joyal, Street, and Verity's $\Int$-construction, and by considering a
strict symmetric compact closed version of Freyd's paracategories.

Some readers may wonder whether we have stated these results at the
right level of generality. It has been suggested that one could start
from partially traced paracategories, or perhaps even partially traced
paramonoidal paracategories, and still get an analogous result.
Indeed, this can probably be done. One can a priori aim for a
representation theorem of the form ``every partially traced
paracategory can be faithfully embedded in a totally traced category,
in such a way that the operations are preserved and reflected''. This
uniquely determines the notion of partially traced paracategory,
namely, they are precisely the reflexive monoidal subgraphs of totally
traced categories. One may then go through the exercise of
axiomatizing this notion. We remark that such an axiomatization is
necessarily quite strange; for example, it can happen that $\Tr([\vec
p\,])$ is defined even when $[\vec p\,]$ is undefined. Whatever
axiomatization one discovers will immediately be rendered obsolete by
the representation theorem, because it is in any case easier to reason
in the larger totally traced category. Thus, in the absence of natural
examples of such paracategories, it is an essentially futile exercise
to try to axiomatize them.

By contrast, the notion of partially traced category, while also made
somewhat obsolete by our representation theorem, is a pre-existing
notion that had been studied in the literature and for which many
interesting examples exist, including some examples that do not {\em
  obviously} arise as subcategories of a totally traced category. Thus
we believe this is indeed a good level of generality.

One question that we did not answer is whether specific partially
traced categories can be embedded in totally traced categories in a
``natural'' way. For example, the category of finite dimensional
vector spaces, with the biproduct $\oplus$ as the tensor, can be
equipped with a natural partial trace in several ways. By our proof,
it follows that it can be faithfully embedded in a totally traced
category. However, we do not know any concrete ``natural'' description
of such a totally traced category (i.e., other than the free one
constructed in our proof).

\bibliographystyle{abbrv}
\bibliography{ptrace}

\end{document}